\documentclass[english,11pt,a4paper]{article}
\usepackage{amsmath,amssymb,amsthm} 

\usepackage{graphicx}
\usepackage[scanall]{psfrag}
\usepackage{color}

\theoremstyle{plain}
\newtheorem{theorem}{Theorem}
\newtheorem{definition}[theorem]{Definition}
\newtheorem{lemma}[theorem]{Lemma}
\newtheorem{proposition}[theorem]{Proposition}
\newtheorem{corollary}[theorem]{Corollary}

\numberwithin{theorem}{section}
\numberwithin{equation}{section}

\newcommand{\R}{{\mathbb{R}}}
\newcommand{\rplus}{{\mathbb{R}^{+}}}

\newcommand{\C}{\mathbb{C}}
\newcommand{\cplus}{\mathbb{C}^{+}}

\newcommand{\N}{\mathbb{N}}

\newcommand{\Ha}[0]{\ensuremath{\mathcal{H}}}
\newcommand{\dif}[0]{\ensuremath{\,\mathrm{d}}}

\renewcommand{\Re}{{\rm Re\,}}

\newcommand{\TF}{{\leftarrow}}

\newcommand{\Xscr} {{\mathcal X}}
\newcommand{\Yscr} {{\mathcal Y}}
\newcommand{\Uscr} {{\mathcal U}}
\newcommand{\Zscr} {{\mathcal Z}}

\newcommand{\ipd}[2]{\langle #1, #2 \rangle}

\newcommand{\bbm}[1]{\left[\begin{matrix} #1 \end{matrix}\right]}
\newcommand{\sbm}[1]{\left[\begin{smallmatrix} #1
             \end{smallmatrix}\right]}

\newcommand{\DisplayNote}[1]{\hspace{.5cm}( {\bf{#1}} ) \hspace{.5cm} }
\newcommand{\OmitThis}[1]{}
\newcommand{\ResearchNote}[1]{}

\newcommand{\ProofNote}[1]{}

\newcommand{\closure}[1]{\overline{#1}}
\newcommand{\abs}[1]{\left \vert #1 \right \vert}

\newcommand{\Dom}[1]{{\rm dom}\left (#1 \right )}

\newcommand{\Null}[1]{{\rm ker} \left (#1 \right )}
\newcommand{\BLO}{\mathcal L}
\newcommand{\norm}[1]{\|{#1}\|}
\newcommand{\rst}[1]{\big{|}_{#1}}

\newcommand{\bg}{\mathbf \gamma}

\newcommand{\bnu}{\mathbf \nu}

\newcommand{\br}{\mathbf r}

\begin{document}

\title{Acoustic wave guides as \\ infinite-dimensional dynamical systems} 
\author{Atte Aalto, Teemu Lukkari, and Jarmo Malinen}

\date{\today}

\maketitle
 
\begin{abstract}   
 We prove the unique solvability, passivity/conservativity and some
 regularity results of two mathematical models for acoustic wave
 propagation in curved, variable diameter tubular structures of finite
 length.  The first of the models is the generalised Webster's model
 that includes dissipation and curvature of the 1D waveguide. The
 second model is the scattering passive, boundary controlled wave
 equation on 3D waveguides. The two models are treated in an unified
 fashion so that the results on the wave equation reduce to the
 corresponding results of approximating Webster's model at the limit
 of vanishing waveguide intersection.
\end{abstract}


\vspace{0.3cm}


\bibliographystyle{plain}


\section{\label{IntroSec} Introduction}

This is the second part of the three part mathematical study on
acoustic wave propagation in a narrow, tubular 3D domain $\Omega
\subset \R^3$. The other parts of the work are
\cite{L-M:PEEWEWP,L-M:WECAD}. Our current interest in wave guide
dynamics stems from modelling of acoustics of speech production; see,
e.g., \cite{A-A-M-V:MLBVFVTO,A-H-K-M-P-S-V:HFAVFFCVTR,H-L-M-P:WFWE}
and the references therein.

The main purpose of the present paper is to give a rigorous treatment
of solvability and energy passivity/conservativity questions of the
two models for wave propagations that are discussed in detail in
\cite{L-M:WECAD}: these are {\rm (i)} the boundary controlled wave
equation on a tubular domain, and {\rm (ii)} the generalised Webster's
horn model that approximates the wave equation in low frequencies. The
\emph{a~posteriori} error estimate for the Webster's model is
ultimately given in \cite{L-M:PEEWEWP}, and it is in an essential part
based on Theorems~\ref{WebsterNodeThm} and \ref{nDimWaveConsThm}
below.

The secondary purpose of this paper is to introduce the new notion of
\emph{conservative majoration} for passive boundary control
systems. The underlying systems theory idea is simple and easy to
explain: it is to be expected on engineering and physical grounds that
adding energy dissipation to a forward time solvable (i.e., internally
well-posed, typically even conservative) system cannot make the system
ill-posed, e.g., unsolvable in forward time direction. Thus, it should
be enough to treat mathematically only the lossless conservative case
that ``majorates'' all models where dissipation is included as far as
we are not reversing the arrow of time.  That this intuition holds
true for many types of energy dissipation is proved in
Theorem~\ref{FirstConservativeMajorationThm} for boundary dissipation
and in Theorem~\ref{SecondConservativeMajorationThm} for a class of
dissipation terms for PDE's. These theorems are given in the general
context of \emph{boundary nodes} that have been discussed in, e.g.,
\cite{M-S:CBCS,M-S:IPCBCS,OS:WPLS}.

Early work concerning Webster's equation can be found in
\cite{EE:CSWHE,VS:GPWHT,VS:NFH,AW:AITHP}. Webster's original work
\cite{AW:AITHP} was published in 1919, but the model itself has a
longer history spanning over 200 years and starting from the works of
D.~Bernoulli, Euler, and Lagrange. More modern approaches is provided
by
\cite{L-L:AMAEMAI,L-L:AMAEMAII,N-T:APDVCS,SR:STSVCALDF,R-E:NCBMSFESSPLFD,SR:WHER}. Webster's
horn model is a special case of the wave equation in a non-homogenous
medium in $\Omega \subset \R^n$, $n \geq 1$, which has been treated
with various boundary and interior point control actions in, e.g.,
\cite[Appendix~2]{F-L-T:ARE}, \cite[Section~2]{L-L-T:NBVPSOHO},
\cite{JLL:ECSPDS}, \cite[Section~6]{DLR:CSTLPDERPOQ}, and, in
particular, \cite[Section~7]{L-T:CTPDECATII} containing also
historical remarks. There exists a rich literature on the damped wave
equation in 1D spatial domain, and instead of trying to give here a
comprehensive account we refer to the numerous references given
\cite{G-H:DSPE}.

The boundary of $\Omega \subset \R^n$, $n \geq 2$, is smooth or $C^2$
in the works cited above, which excludes polygons (for $n = 2$) or
their higher dimensional counterparts such as the tubular structures
discussed here.  From systems theory point of view, this is a serious
restriction since it is obviously impossible to connect finitely many,
disjoint, smooth domains seamlessly to each other without leaving
holes whose interior is non-empty.  The generality of this article
makes it possible to interconnect 3D wave equation systems on
geometrically compatible elements $\Omega_j \subset \R^3$ to form
aggregated systems on $\cup_j \Omega_j$ in the same way as described
in \cite[Section~5]{A-M:CPBCS} for Webster's horn model.
 \begin{figure}
 \centering
 \def\svgwidth{0.45\textwidth}
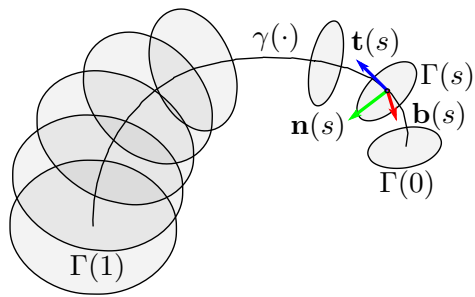 
 \caption{The Frenet frame of the planar centreline for a tubular
   domain $\Omega$, represented by some of its intersection surfaces
   $\Gamma(s)$ for $s \in [0,1]$. The wall $\Gamma \subset \partial \Omega$ is not
   shown, and the global coordinate system is detailed in
   \cite[Section~2]{L-M:WECAD}. }\label{centrelines}
 \end{figure}

Theorems~\ref{WebsterNodeThm} and \ref{nDimWaveConsThm} treat the
questions of unique solvability, passivity, and regularity of the two
wave propagation models in the exactly same form as these results are
required in companion papers \cite{L-M:PEEWEWP,L-M:WECAD}.  The strict
passivity (i.e., the case $\alpha > 0$) in
Theorems~\ref{WebsterNodeThm} and \ref{nDimWaveConsThm} could be
proved without resorting to
Theorems~\ref{FirstConservativeMajorationThm} and
\ref{SecondConservativeMajorationThm} as they both concern single PDE's
with simple dissipation models. However, the direct approach becomes
technically quite cumbersome if we have more complicated aggregated
systems to treat (not all of which need be defined by PDE's), and
combinations of various dissipation models are involved. An example of
such systems is provided by \emph{transmission graphs} as introduced
in \cite{A-M:CPBCS}
where the general passive case is treated by reducing it to the
conservative case and arguing as in
Theorem~\ref{SecondConservativeMajorationThm}. In the context of
transmission graphs, see also the literature on port-Hamiltonian
systems \cite{C-S-N:IPHSCDS,K-Z-S-B:DSTCHS,JV:PHADPS}.  That the
conservative majoration method cannot be used for all possible
dissipation terms is shown in Section~\ref{ConclSec} by an example
involving Kelvin--Voigt structural damping.

 Let us return to wave propagation models on a tubular domain $\Omega$
 referring to Fig.~\ref{centrelines}.  The cross sections $\Gamma(s)$
 of $\Omega$ are normal to the planar curve $\bg = \bg(s)$ that serves
 as the centreline of $\Omega$ as shown in Fig.~\ref{centrelines}.  We
 denote by $R(s)$ and $A(s) := \pi R(s)^2$ the radius and the area of
 $\Gamma(s)$, respectively.  We call $\Gamma$ the \emph{wall}, and the
 circular plates $\Gamma(0)$, $\Gamma(1)$ the \emph{ends} of the tube
 $\Omega$. The boundary of $\Omega$ satisfies $\partial \Omega =
 \closure{\Gamma} \cup \Gamma(0) \cup \Gamma(1)$. Without loss of
 generality, the parameter $s \geq 0$ can be regarded as the arc
 length of $\bg$, measured from the control/observation surface
 $\Gamma(0)$ of the tube.

As is well known, acoustic wave propagation in $\Omega$ can be
modelled by the wave equation for the velocity potential $\phi$  as
\begin{equation}\label{IntroWaveEq}
\begin{cases}
  &  \phi_{tt}(\br, t) = c^2 \Delta \phi(\br, t) \quad
  \text{ for } \br \in \Omega \text{ and } t \in \rplus, \\
  &   c \frac{\partial \phi}{\partial
    \bnu}(\br, t) + \phi_t(\br, t) = 2 \sqrt{\tfrac{c}{\rho A(0)}} \, u(\br, t) \quad
  \text{ for } \br \in \Gamma(0) \text{ and }  t \in \rplus, \\
  &  \phi(\br, t)  = 0 \quad \text{ for } \br \in \Gamma(1) \text{ and }
  t \in \rplus, \\
  & \alpha \frac{\partial \phi}{\partial t}(\br, t) 
  + \frac{\partial \phi}{\partial \bnu}(\br, t)  = 0 \quad \text{ for } \br
  \in \Gamma, \text{ and }
  t \in \rplus, \text{ and } \\
  & \phi(\br,0)  = \phi_0(\br), \quad \rho \phi_t(\br,0) = p_0(\br) \quad \text{ for
  } \br \in \Omega
\end{cases}
\end{equation}
with the observation defined by
\begin{equation}\label{IntroWaveObs}
  c \frac{\partial \phi}{\partial
    \bnu} (\br, t) - \phi_t(\br, t)  =  2 \sqrt{\tfrac{c}{\rho A(0)}} \, 
  y(\br, t)  \quad
  \text{ for }  \br \in \Gamma(0) \text{ and }  t \in \rplus,
\end{equation}
where $\bnu$ denotes the unit normal vector on $\partial \Omega$, $c$
is the sound speed, $\rho$ is the density of the medium, and $\alpha
\geq 0$ is a parameter associated to boundary dissipation. The
functions $u$ and $y$ are control and observation signals in
\emph{scattering form}, and the normalisation constant $2
\sqrt{\tfrac{c}{\rho A(0)}}$ takes care of their physical dimension
which is power per area. Solvability, stability, and energy questions
for the wave equation in various geometrical domains $\Omega \subset
\R^n$ have a huge literature, and it is not possible to give a
historically accurate review here. The wave equation is a prototypal
example of a linear hyperbolic PDE whose classical mathematical
treatment can be found, e.g., in \cite[Chapter~5]{L-M:NHBVPAPII}, and
the underlying physics is explained well in
\cite[Chapter~9]{F-W:TMPAC}. In the operator and mathematical system
theory context, it has been given as an example (in various
variations) in
\cite{JM:CTFIBCS,M-S:IPCBCS,RT:WEBDBDOA,W-T:HTGCWPLSOOTAII,W-T:HTGCWPLSOOTA}
and elsewhere. For applications in speech research, see, e.g.,
\cite{A-H-K-M-P-S-V:HFAVFFCVTR,H-L-M-P:WFWE,L-M:WECAD} and the
references therein.

One computationally and analytically simpler wave propagation model is
the \emph{generalised Webster's horn model} for the same tubular
domain $\Omega$ that is now represented by the \emph{area function}
$A(\cdot)$ introduced above.
To review this model in its generalised form, let us recall some
notions from \cite{L-M:WECAD}. To take into account the curvature
$\kappa(s)$ of the centreline $\bg(\cdot)$ of $\Omega$, we adjust the
sound speed $c$ in \eqref{IntroWaveEq} by defining $c(s) := c
\Sigma(s)$ where $\Sigma(s) := \left ( 1 + \tfrac{1}{4} \eta(s)^2
\right )^{-1/2}$ is the \emph{sound speed correction factor}, and
$\eta(s) := R(s) \kappa(s)$ is the \emph{curvature ratio} at $s \in
[0,1]$.  We also need take into consideration the deformation
of the outer wall $\Gamma$ by defining the \emph{stretching
  factor} $W(s) := R(s)\sqrt{R'(s)^2+(\eta(s) - 1)^2}$; see
\cite[Eq.~(2.8)]{L-M:WECAD}.  It is a standing assumption that
$\eta(s) < 1$ to prevent the tube $\Omega$ from folding on
itself locally.

Following \cite{L-M:WECAD}, the generalised Webster's horn model for
the velocity potential $\psi = \psi(s,t)$ is now given by
\begin{equation} \label{IntroWebstersEq}
  \begin{cases}
    & \psi_{tt} = \frac{c(s)^{2}}{A(s)} \frac{\partial}{\partial s}
    \left ( A(s) \frac{\partial \psi }{\partial s} \right ) 
    - \frac{2 \pi \alpha W(s) c(s)^{2} }{A(s)} \frac{\partial \psi }{\partial t}
\\   & \hfill \text{ for }   s \in (0,1) \text{ and } t \in \rplus, \\
    & - c \psi_s(0,t) + \psi_t(0,t) = 2 \sqrt{\frac{c}{\rho
        A(0)}} \, \tilde u(t) \quad
    \text{ for } t \in \rplus,  \\
    & \psi(1,t) = 0 \quad  \text{ for } t \in \rplus, \quad \text{ and } \\
    & \psi(s,0) = \psi_0(s), \quad \rho \psi_t(s,0) = \pi_0(s) \quad
    \text{ for } s \in (0,1),
\end{cases}
\end{equation}
and the observation $\tilde y$ is defined by
\begin{equation}\label{IntroWebstersEqObs}
   - c \psi_s(0,t) - \psi_t(0,t) = 2 \sqrt{\frac{c}{\rho
        A(0)}} \, \tilde y(t) \quad
    \text{ for } t \in \rplus.
\end{equation}
The constants $c$, $\rho$, $\alpha$ are same as in
\eqref{IntroWaveEq}. The input and output signals $\tilde u$ and
$\tilde y$ of \eqref{IntroWebstersEq}--\eqref{IntroWebstersEqObs}
correspond to $u$ and $y$ in \eqref{IntroWaveEq}--\eqref{IntroWaveObs}
by spatial averaging over the control surface $\Gamma(0)$.  Hence,
their physical dimension is power per area as well.  Based on
\cite{L-M:PEEWEWP,L-M:WECAD}, the solution $\psi$ of
\eqref{IntroWebstersEq} approximates the averages
\begin{equation} \label{AveragedWaveEqSol}
 \bar \phi(s,t) : = \frac{1}{A(s)} {\int_{ \Gamma(s) }{\phi \dif A}} \quad \text{ for } \quad 
s \in (0,1) \quad \text{ and }  \quad  t \geq 0 
\end{equation}
of $\phi$ in \eqref{IntroWaveEq} when $\phi$ is regular enough.  Note
that the dissipative boundary condition $\alpha \frac{\partial
  \phi}{\partial \bnu}(\br, t) + \frac{\partial \phi}{\partial
  \bnu}(\br, t) = 0$ in \eqref{IntroWaveEq} has been replaced by the
dissipation term $2 \pi \alpha W(s)  A(s)^{-1} c(s)^{2}
\frac{\partial \psi }{\partial t}$ (with the same parameter $\alpha$)
in \eqref{IntroWebstersEq}. For classical work on Webster's horn
model, see \cite{L-L:AMAEMAI,N-T:APDVCS,VS:GPWHT} and in particular
\cite{SR:WHER} where numerous references can be found.

We show in Theorem~\ref{nDimWaveConsThm} that the wave equation model
\eqref{IntroWaveEq}--\eqref{IntroWaveObs} is uniquely solvable in both
directions of time, and the solution satisfies an energy inequality if
$\alpha > 0$.  By Corollary~\ref{nDimWaveConsThmCor}, the model has
the same properties for $\alpha = 0$ but then the energy inequality is
replaced by an equality, and the model is even time-flow
invertible. In all cases, the solution $\phi$ is observed to have the
regularity required for the treatment given in \cite{L-M:WECAD} if the
input $u$ is twice continuously differentiable. The generalised
Webster's horn model
\eqref{IntroWebstersEq}--\eqref{IntroWebstersEqObs} is treated in a
similar manner in Theorem~\ref{WebsterNodeThm}.

This paper is organised as follows: Background on boundary control
systems is given in Section~\ref{BackgroundSection}.  Conservative
majoration of passive boundary control systems is treated in
Section~\ref{MajorationSection}.  The Webster's horn model and the
wave equation are treated in Sections~\ref{WebsterSec} and
\ref{WaveSection} respectively.  Some immediate extensions of these
results are given in Section~\ref{ConclSec}.  Because of the lack of
accessible, complete, and sufficiently general references, the paper
is completed by a self-contained appendix on Sobolev spaces, boundary
trace operators, Green's identity, and Poincar\'e inequality for
special Lipschitz domains that are required in the rigorous analysis
of typical wave guide geometries.

\section{\label{BackgroundSection} On infinite dimensional systems}

Linear boundary control systems such as \eqref{IntroWaveEq} and
\eqref{IntroWebstersEq} are treated as dynamical systems that can be
described by operator differential equations of the form
\begin{equation} \label{BckGrndDynEq}
    u(t) = Gz(t), \quad \dot z(t) = L z(t), \quad 
    \text{ with the initial condition } \quad z(0) = z_0
\end{equation}
and the observation equation
\begin{equation}
  \label{BckGrndDynEqObs}
  y(t) = Kz(t), 
\end{equation}
where $t \in \rplus$ denotes time.  The signals in
\eqref{BckGrndDynEq}, \eqref{BckGrndDynEqObs} are as follows: $u$ is
the input, $y$ is the output, and the state trajectory is $z$.

\subsubsection*{Cauchy problems}
To make \eqref{BckGrndDynEq} properly solvable for all twice
differentiable $u$ and compatible initial states $z_0$, the axioms of
an \emph{internally well-posed boundary node} should be satisfied:

\begin{definition} \label{BoundaryNodeDef} A triple of operators $\Xi = (G,L,K)$
  is an \emph{internally well-posed boundary node} on the Hilbert
  spaces $(\Uscr,\Xscr,\Yscr)$ if the following conditions are
  satisfied:

\begin{enumerate}
  
\item  $G$, $L$, and $K$ are linear operators with the same domain
  $\Zscr \subset \Xscr$;
 
\item \label{BoundaryNodeDefRequirement2} $\sbm{G \\ L \\ K}$ is a closed operator from
 $\Xscr$  into $\Uscr \times \Xscr \times \Yscr$ with domain $\Zscr$;

\item $G$ is surjective, and $\Null{G}$ is dense in $\Xscr$; and

\item \label{IntWellPosedBNDNode} $L\rst{\Null{G}}$ (understood as an unbounded operator in $\Xscr$
  with domain $\Null{G}$) generates a strongly continuous semigroup on
  $\Xscr$.
\end{enumerate}
If, in addition, $L$ is a closed operator on $\Xscr$ with domain
$\Zscr$, we say that the boundary node $\Xi$ is \emph{strong}.
\end{definition}
The history of abstract boundary control system dates back to
\cite{HOF:BCS,DS:IDL,DS:RTH}. The phrase ``internally well-posed''
refers to condition \eqref{IntWellPosedBNDNode} of
Definition~\ref{BoundaryNodeDef}, and it is a much weaker property
than well-posedness of systems in the sense of \cite{OS:WPLS}.  It
plainly means that the boundary node defines an evolution equation
that is uniquely solvable in forward time direction.  Boundary nodes
that are not necessarily internally well-posed are characterised by the
weaker requirement in place of \eqref{IntWellPosedBNDNode}: $\alpha -
L\rst{\Null{G}}$ is a bijection from $\Null{G}$ onto $\Xscr$ for some
$\alpha \in \C$.

We call $\Uscr$ the \emph{input space}, $\Xscr$ the \emph{state
  space}, $\Yscr$ the \emph{output space}, $\Zscr$ the \emph{solution
  space}, $G$ the \emph{input boundary operator}, $L$ the
\emph{interior operator}, and $K$ the \emph{output boundary operator}.
The operator $A := L\rst{\Null{G}}$ is called the \emph{semigroup
  generator} if $\Xi$ is internally well-posed, and otherwise it is
known as the \emph{main operator} of $\Xi$.  Because $\bbm{G & L &
  K}^T$ is a closed operator, we can give its domain the Hilbert space
structure by the \emph{graph norm}
\begin{equation} \label{SolutionSpaceNorm}
 \norm{z}_\Zscr^2 =  \norm{z}_\Xscr^2 + \norm{Lz}_\Xscr^2 + \norm{Gz}_\Uscr^2 + \norm{Kz}_\Yscr^2.
\end{equation}
If the node is strong, we have an equivalent norm for $\Zscr$ given by
omitting the last two terms in \eqref{SolutionSpaceNorm}.  If $\Xi =
(G,L,K)$ is an internally well-posed boundary node, then
\eqref{BckGrndDynEq} has a unique ``smooth'' solution:
\begin{proposition} \label{SolvabilityProp} Assume that $\Xi =
  (G,L,K)$ is an internally well-posed boundary node.  For all $z_0
  \in \Xscr$ and $u \in C^2(\rplus;\Uscr)$ with $Gz_0 = u(0)$ the
  equations \eqref{BckGrndDynEq} have a unique solution $z \in
  C^1(\rplus;\Xscr) \cap C(\rplus;\Zscr)$.  Hence, the output $y \in
  C(\rplus;\Yscr)$ is well defined by the equation~\eqref{BckGrndDynEqObs}.
\end{proposition}
\noindent Indeed, this is \cite[Lemma 2.6]{M-S:CBCS}.

\subsubsection*{Energy balances}
Now that we have treated the solvability of the dynamical equations,
it remains to consider energy notions. We say that the internally
well-posed boundary node $\Xi = (G,L,K)$ is \emph{(scattering)
  passive} if all smooth solutions of \eqref{BckGrndDynEq} satisfy
\begin{equation}
 \label{ConsCondDiff} \frac{d}{dt} \norm{z(t)}_\Xscr^2
 + \norm{y(t)}_\Yscr^2 \leq \norm{u(t)}_\Uscr^2 \quad \text{ for all } \quad t \in \rplus
\end{equation}
with $y$ given by \eqref{BckGrndDynEqObs}.  All such systems are
well-posed in the sense of \cite{OS:WPLS}; see also \cite{T-W:OCOS}.
We say that $\Xi$ is \emph{(scattering) energy preserving} if
\eqref{ConsCondDiff} holds as an equality.

Many boundary nodes arising from hyperbolic PDE's (such as
\eqref{IntroWaveEq}--\eqref{IntroWaveObs} and
\eqref{IntroWebstersEq}--\eqref{IntroWebstersEqObs}) have the property
that they remain boundary nodes if we {\rm (i)} change the sign of $L$
(i.e., reverse the direction of time); and {\rm (ii)} interchange the
roles of $K$ and $G$ (i.e., reverse the flow direction). Such boundary
nodes are called \emph{time-flow invertible}, and we write $\Xi^\TF =
(K, -L, G)$ for the time-flow inverse of $\Xi$. There are many
equivalent definitions of \emph{conservativity} in the literature, and
we choose here the following:
\begin{definition}\label{ConsNodeDef} An internally well-posed boundary node $\Xi$
is \emph{(scattering) conservative} if it is time-flow invertible, and
both $\Xi$ itself and the time-flow inverse $\Xi^\TF$ are (scattering)
energy preserving.\footnote{The words ``energy preserving'' can be
  replaced by ``passive'' without changing the class of systems one obtains.}
\end{definition}
For system nodes that have been introduced in \cite{OS:WPLS,M-S-W:HTCCS}, an
equivalent definition for conservativity is to require that both $S$
and its \emph{dual node} $S^d$ are energy preserving. This is the
straightforward generalisation from the finite-dimensional theory but
it is not very practical when dealing with boundary control. For
conservative systems, the time-flow inverse and the dual system
coincide, and we have then, in particular, $A^* = - L\rst{\Null{K}}$
if $A = L\rst{\Null{G}}$. For details, see \cite[Theorems~1.7 and
  1.9]{M-S:CBCS}.

It is possible to check economically, without directly using
Definition \ref{BoundaryNodeDef}, that the triple $\Xi = (G,L,K)$ is a
dissipative/conservative boundary node:
\begin{proposition}
\label{CheckingPassivityProp}  Let $\Xi = (G,L,K)$ be a 
triple of linear operators with a common domain $\Zscr \subset \Xscr$,
and ranges in the Hilbert spaces $\Uscr$, $\Xscr$, and $\Yscr$,
respectively.  Then $\Xi$ is a passive boundary node on
$(\Uscr,\Xscr,\Yscr)$ if and only if the following conditions hold:
\begin{enumerate}

\item \label{PassCharCond3} We have the \emph{Green--Lagrange inequality}
  \begin{equation} \label{PassScattIneq}
    2 \Re \ipd{z}{Lz}_\Xscr + \norm{K z}_\Yscr^2 \leq \norm{G z}_\Uscr^2
\quad \text{ for all } \quad z \in \Zscr;
  \end{equation}

\item \label{PassCharCond2} 
$G \Zscr = \Uscr$ and $(\beta - L) \Null{G} = \Xscr$
 for some $\beta \in \cplus$ (hence, for all $\beta \in \cplus$).

\end{enumerate}

Similarly, $\Xi$ is a conservative boundary node on
$(\Uscr,\Xscr,\Yscr)$ if and only if \eqref{PassCharCond2} above holds
together with the additional conditions:
\begin{enumerate}
\addtocounter{enumi}2

\item \label{PassCharCond1} 

We have the  \emph{Green--Lagrange identity}
\begin{equation}\label{BackgroundG-L}
 2 \Re \ipd z{L z}_\Xscr + \norm{Kz}_\Yscr^2
 = \norm{Gz}_\Uscr^2  \quad \text{ for all } \quad z \in \Zscr.
\end{equation}

\item \label{PassCharCond4}

$K \Zscr = \Yscr$ and $(\gamma + L) \Null{K} = \Xscr$ for some $\gamma
  \in \cplus$ (hence, for all $\gamma \in \cplus$).


\end{enumerate}

\end{proposition}
\noindent This is a slight modification of \cite[Theorem
  2.5]{M-S:IPCBCS}. See also \cite[Proposition 2.5]{M-S:CBCS}. The
abstract boundary spaces as discussed in \cite{G-G:BVPODE} are
essentially (impedance) conservative strong nodes as explained in
\cite[Section~5]{M-S:IPCBCS}.

\section{\label{MajorationSection} Conservative majorants} 

In some applications, the dissipative character of a linear dynamical
system is often due to a distinct part of the model such as a term or
a boundary condition imposed on the defining PDE.  If this part is
completely removed from the model, the resulting more simple system is
conservative and, in particular, internally well-posed. We call it a
\emph{conservative majorant} of the original dissipative system.

Intuition from engineering and physics hints that increasing
dissipation should make the system ``better behaved'' and not spoil
the internal well-posed\-ness.\footnote{The dissipativity or even the
  internal well-posedness of the time-flow inverted system is, if
  course, destroyed since adding dissipation creates the ``arrow of
  time''.}  The following
Theorems~\ref{FirstConservativeMajorationThm} and
\ref{SecondConservativeMajorationThm} apply to many boundary control
systems. However, they are written for \emph{passive} majorants since
the proofs remain the same, and this way the results can be applied
successively to systems having both boundary dissipation and
dissipative terms.  
\begin{theorem} \label{FirstConservativeMajorationThm}
  Let $\widetilde \Xi = (\sbm{G \\ \tilde G},L,\sbm{K \\ \tilde K})$
  be a scattering passive boundary node on Hilbert spaces $(\Uscr
  \oplus \tilde \Uscr,\Xscr, \Yscr \oplus \tilde \Yscr)$ with solution
  space $\tilde \Zscr$.  Then $\Xi := (G\rst{\Zscr}, L\rst{\Zscr},
  K\rst{\Zscr})$ is a scattering passive boundary node on
  $(\Uscr,\Xscr,\Yscr)$ with the solution space $\Zscr := \Null{\tilde
    G}$. Both $\widetilde \Xi$ and $\Xi$ have the same semigroup
  generators, equalling $L\rst{\Null{G} \cap \Null{\tilde G}}$.  If
  $\widetilde \Xi$ is a strong node, so is $\Xi$.
\end{theorem}
\begin{proof}
 The Green--Lagrange inequality holds for $\Xi$ since for $z \in
 \Null{\tilde G}$ we have $\norm{G z}_{\Uscr} = \norm{\sbm{G \\ \tilde
     G}z}_{\Uscr \oplus \tilde \Uscr}$, and hence we get by the
 passivity of $\widetilde \Xi$
\begin{equation*}
     2 \Re \ipd{z}{L z}_\Xscr - \norm{G z}_{\Uscr}^2
     \leq - \norm{\sbm{K z \\ \tilde  K z}}_{\Yscr \oplus \tilde \Yscr}^2
     \leq - \norm{Kz}_{\Yscr}^2.
\end{equation*}
The surjectivity $G \Zscr = \Uscr$ follows from
$\Uscr \oplus \{ 0 \} \subset \Uscr \oplus \tilde \Uscr = \sbm{G
  \\ \tilde G} \Zscr$ and $\Zscr = \Null{\tilde G}$.
 Since $(\beta - L) \Null{G\rst{\Zscr}} = (\beta - L)\rst{\Null{\tilde G}} \Null{G} =
(\beta - L) \left ( \Null{G} \cap \Null{\tilde G} \right ) = (\beta - L) \Null{\sbm{G
    \\ \tilde G}} = \Xscr$, the passivity of $\Xi$ follows by
Proposition \ref{CheckingPassivityProp}.

Suppose that $L$ is closed (i.e., $\widetilde \Xi$ is strong) and that
$\tilde \Zscr \supset \Zscr \ni z_j \to z$ in $\Xscr$ is such that $L
z_j \to x$ in $\Xscr$ as $j \to \infty$. Because $L$ is closed, $z \in
\Dom{L} = \tilde \Zscr$ and $Lz = x$.  Thus, $\norm{z_j - z}_{\Zscr}^2
:= \norm{z_j - z}_{\Xscr}^2 + \norm{L(z_j - z)}_{\Xscr}^2 \to
0$. Because $\tilde G \in \BLO(\Zscr;\tilde \Uscr)$ by applying 
\eqref{SolutionSpaceNorm} on $\widetilde \Xi$, the space $\Zscr = \Null{\tilde G}$ is
closed in $\tilde \Zscr$ and thus $z \in \Zscr$.    We have now shown
that $L\rst{\Zscr}$ is closed with $\Dom{L\rst{\Zscr}} = \Zscr$.
\end{proof}
The restriction of the original solution space to $\Null{\tilde G}$ in
Theorem~\ref{FirstConservativeMajorationThm} is a functional analytic
description of boundary dissipation of a particular kind. If the
original scattering passive $\widetilde \Xi$ is translated to an
impedance passive boundary node by the external Cayley-transform (see
\cite[Definition~3.1]{M-S:IPCBCS}), then the abstract boundary
condition by restriction to $\Null{\tilde G}$ can be understood as a
termination to an ideally resistive element as depicted in
\cite[Fig.~1]{M-S:IPCBCS}.
\begin{theorem}
  \label{SecondConservativeMajorationThm}
  Let $\Xi = (G,L,K)$ be a scattering passive boundary node on Hilbert
  spaces $(\Uscr,\Xscr,\Yscr)$ with solution space $\Zscr$ and
  $\Xscr_1 = \Null{G}$ with the norm $\norm{z}_{\Xscr_1} = \norm{(1 -
    L)z}_{\Xscr}$.  Let $H$ be a dissipative operator on $\Xscr$ with
  $\Zscr \subset \Dom{H}$.\footnote{This means that $H:\Dom{H} \subset
    \Xscr \to \Xscr$ is an operator satisfying $\Zscr \subset \Dom{H}$
    and $\Re \ipd{z}{Hz}_\Xscr \leq 0$ for all $z \in \Zscr$.}  Denote
  the two assumptions as follows:
  \begin{enumerate}
  \item \label{SecondConservativeMajorationThmAss1} 
    There is $a > 0$ and $0 \leq b < 1$ such that $\norm{Hz}_\Xscr \leq a
    \norm{z}_\Xscr + b \norm{Lz}_\Xscr$ for all $z \in \Null{G}$.
  \item \label{SecondConservativeMajorationThmAss2} There is a Hilbert
    space $\tilde \Xscr$ such that $\Xscr_1 \subset \tilde \Xscr \subset
    \Dom{H}$, the inclusion $\Xscr_1 \subset \tilde \Xscr$ is compact
    and $H\rst{\tilde \Xscr} \in \BLO(\tilde \Xscr; \Xscr)$.
  \end{enumerate}
  If either \eqref{SecondConservativeMajorationThmAss1} or
  \eqref{SecondConservativeMajorationThmAss2} holds, then $\Xi_H :=
  (G,L+H,K)$ is a scattering passive boundary node.  We have $\Dom{A}
  = \Dom{A_H}$ where $A = L\rst{\Null{G}}$ and $A_H =
  (L+H)\rst{\Null{G}}$ are the semigroup generators of $\Xi$ and
  $\Xi_H$, respectively.  If the node $\Xi$ is strong and $H \in
  \BLO(\Xscr)$ (i.e., $b = 0$ in assumption
  \eqref{SecondConservativeMajorationThmAss1}), then $\Xi_H$ is a
  strong boundary node as well.
\end{theorem}
\noindent Both the assumptions
\eqref{SecondConservativeMajorationThmAss1} and
\eqref{SecondConservativeMajorationThmAss2} hold if $H \in
\BLO(\Xscr)$ and $\Xscr_1 \subset \Xscr$ with a compact inclusion.
This is the case in \cite[Section~5]{A-M:CPBCS} in the context of an
impedance passive system.  The compactness property is typically a
consequence of the Rellich--Kondrachov theorem \cite[Theorem 1,
  p. 144]{E-G:MTFPF} for boundary nodes defined by PDE's on bounded
domains.  In many applications such as Theorem~\ref{WebsterNodeThm}
below, the operator $H$ is even self-adjoint.  We give an example of
the 1D wave equation with Kelvin--Voigt damping in Section~\ref{ConclSec}
  where Theorem~\ref{SecondConservativeMajorationThm} cannot be
  applied.
\begin{proof}
By using assumption \eqref{SecondConservativeMajorationThmAss1}: This
argument is motivated by \cite[Theorem~2.7 on p.~501]{TK:PTLO}.  Let
us first show that $A_H := A + H\rst{\Null{G}}$ with $\Dom{A_H} =
\Null{G}$ generates a contraction semigroup on $\Xscr$ where $A =
L\rst{\Null{G}}$ generates the contraction semigroup of $\Xi$ as
usual. As a first step, we establish the inequality $\norm{H (s -
  A)^{-1}}_{\BLO(\Xscr)} < 1$ for all real $s$ large enough.

Let $\beta > 0$ be arbitrary. For all $s > \beta$ and $z \in \Xscr$ we
have
\begin{equation} \label{SecondConservativeMajorationThmAss1Eq1}
\begin{aligned}
  \norm{H (s - A)^{-1} z}_\Xscr \leq & a \norm{(s - A)^{-1}z}_\Xscr +
  b \norm{A(s - A)^{-1}z}_\Xscr \\  \leq & (a + \beta b) \norm{(s -
    A)^{-1}z}_\Xscr\\& + \frac{b}{s - \beta} \left \| {\left (\frac{1}{s -
      \beta} - (A - \beta)^{-1} \right )^{-1}z} \right \|_\Xscr
\end{aligned}
\end{equation}
since
\begin{equation*}
- A (s - A)^{-1} = \frac{1}{s - \beta} \left (\frac{1}{s - \beta} - (A -
\beta)^{-1} \right )^{-1} - \beta (s - A)^{-1}.
\end{equation*}
Since $A$ is a maximally dissipative operator on $\Xscr$, we have for
all $z = (A - \beta) x \in \Xscr$ with $x \in \Dom{A}$
\begin{align*}
   \Re \left< (A - \beta)^{-1} z, z \right >_{\Xscr} = &\Re \left< (A
  - \beta)^{-1} (A - \beta) x, (A - \beta) x \right >_{\Xscr} \\ 
   = & \Re \left< x, (A - \beta) x \right >_{\Xscr}\\ =& \Re \left< x, A x
  \right >_{\Xscr} - \beta \norm{x}^2_\Xscr \leq 0.
\end{align*}
Thus, the operator $(A - \beta)^{-1}$ is dissipative, and it is
maximally so because $(A - \beta)^{-1} \in \BLO(\Xscr)$.

Because $(A - \beta)^{-1}$ generates a $C_0$ contraction semigroup on
$X$, the Hille--Yoshida generator theorem gives the resolvent estimate
\begin{equation*}
  \frac{1}{s - \beta} \left \| \left (\frac{1}{s - \beta} - (A -
    \beta)^{-1} \right )^{-1} \right \|_{\BLO(\Xscr)} \leq 1
\end{equation*}
for $s > \beta > 0$. Similarly, $\norm{(s - A)^{-1}}_{\BLO(\Xscr)}
\leq 1/s$ for $s > 0$. These together with
\eqref{SecondConservativeMajorationThmAss1Eq1} give
\begin{equation*}
\frac{\norm{H (s - A)^{-1} z}_\Xscr}{\norm{z}_\Xscr} \leq \frac{a + \beta b}{s}
+ b <  1 \text{ for all } s > \frac{a+ \beta b}{1 - b}.
\end{equation*}
Because $\beta > 0$ was arbitrary, we get $\norm{H (s -
  A)^{-1}}_{\BLO(\Xscr)} < 1$ for all $s > \frac{a}{1 - b}$.  We
conclude that $(a/(1 - b), \infty) \subset \rho(A_H)$ and
\begin{equation} \label{SecondConservativeMajorationThmAss1Eq2}
(s - A_H)^{-1} =  (s - A)^{-1} (I - H(s - A)^{-1})^{-1}
\end{equation}
 where $\Dom{A_H} =
\Dom{A} = \Null{G}$.  In particular, we have shown that $(2a/(1 - b) -
L - H)\Null{G} = \Xscr$ (that $G \Zscr = \Uscr$ holds, follows because
$\Xi$ itself is a boundary node with the same input boundary operator
$G$).  Since the Green--Lagrange inequality \eqref{PassScattIneq}
holds by the passivity of $\Xi$ and $\Re \ipd{z}{Hz}_\Xscr \leq 0$ by
assumption, we conclude that \eqref{PassScattIneq} holds with $L + H$
in place of $L$, too.  Thus $\Xi_H$ is a scattering passive boundary
node by Proposition \ref{CheckingPassivityProp}.

By using assumption \eqref{SecondConservativeMajorationThmAss2}: As in
the first part of this proof, it is enough to prove that $\rho(A_H)
\cap \C_+ \neq \emptyset$ by verifying
\eqref{SecondConservativeMajorationThmAss1Eq2}. Because $(s - A)^{-1}
\in \BLO(\Xscr;\Xscr_1)$, $\Xscr_1 \subset \tilde \Xscr$ is compact,
and $H\rst{\tilde \Xscr} \in \BLO(\tilde \Xscr;\Xscr)$, we conclude
that $H(s - A)^{-1} \in \BLO(\Xscr)$ is a compact operator for all $s
\in \C_+$.  If there is a $s > 0$ such that $1 \notin \sigma(H(s -
A)^{-1}) \subset \sigma_p(H(s - A)^{-1}) \cup \{ 0 \}$, then
\eqref{SecondConservativeMajorationThmAss1Eq2} holds, $s \in
\rho(A_H)$, and $\Xi_H$ is a passive boundary node as argued in the
first part of the proof. For contradiction, assume that $1 \in
\sigma_p(H(s_0 - A)^{-1})$ for some $s_0 > 0$. This implies $A_H x_0 =
s_0 x_0$ for some $x_0 \in \Dom{A_H}$, and hence
\begin{equation*}
  \Re \left <A_H x_0, x_0 \right >_\Xscr = s_0 \norm{x_0}^2_\Xscr > 0
\end{equation*}
which contradicts the dissipativity of $A_H = A + H\rst{\Null{G}}$.
Thus \eqref{SecondConservativeMajorationThmAss1Eq2} holds and $\Dom{A}
= \Dom{A_H}$. The final claim about strongness of $\Xi_H$ holds
because perturbations of closed operators by bounded operators are
closed.
\end{proof} 
The perturbation $H$ in Theorem~\ref{SecondConservativeMajorationThm}
is a densely defined dissipative operator on $\Xscr$. As such, it has
a maximally dissipative (closed) extension
$\widetilde{H}:\Dom{\widetilde{H}} \subset \Xscr \to \Xscr$ satisfying
$\widetilde{H}^* \subset H^*$, and the adjoint $\widetilde{H}^*$ is
maximally dissipative as well.  Without loss of generality we may
assume that $H = \widetilde{H}$ in
Theorem~\ref{SecondConservativeMajorationThm}.  Furthermore, it is
possible to use $\tilde \Xscr = \Dom{\widetilde{H}}$ equipped with the
graph norm $\norm{z}^2_{\Dom{\widetilde{H}}} = \norm{z}_\Xscr^2 +
\norm{\widetilde{H} z}_\Xscr^2$ in assumption
\eqref{SecondConservativeMajorationThmAss2}, and it only remains to
check whether $\Xscr_1 \subset \Dom{\widetilde{H}}$ compactly.

Let us consider the adjoint semigroup of the passive boundary node
$\Xi_H = (G, L+H, K)$, majorated by the conservative node $\Xi =
(G,L,K)$. The adjoint semigroup is generated by the maximally
dissipative operator $A_H^*$ where $A_H = (L + H)\rst{\Null{G}}$ is
maximally dissipative under the assumptions of
Theorem~\ref{SecondConservativeMajorationThm}.
\begin{proposition} \label{AdjointSemigroupProp}
Let $\Xi = (G,L,K)$ be a scattering conservative boundary node on
Hilbert spaces $(\Uscr,\Xscr,\Yscr)$ with solution space $\Zscr$.  Let
$H$ be a dissipative operator on $\Xscr$ with $\Zscr \subset \Dom{H}$.
Assume that either of the assumptions
\eqref{SecondConservativeMajorationThmAss1} or
\eqref{SecondConservativeMajorationThmAss2} of
Theorem~\ref{SecondConservativeMajorationThm} holds, and let the
extension $\widetilde{H}$ be defined as above.
\begin{enumerate}
\item \label{AdjointSemigroupPropClaim1} If $\Null{K} \subset
  \Dom{\widetilde{H}^*}$, then $(- L + \widetilde{H}^*)\rst{\Null{K}}
  \subset A_H^*$.
\item \label{AdjointSemigroupPropClaim2} If $\Xi$ is time-flow
  invertible and $\Zscr \subset \Dom{\widetilde{H}^*}$, then
  $\Xi^\TF_{\widetilde{H}^*} := (K, -L + \widetilde{H}^*,G)$ is an
  internally well-posed boundary node if and only if $(- L +
  \widetilde{H}^*)\rst{\Null{K}} = A_H^*$.
\item \label{AdjointSemigroupPropClaim3} If $\Xi$ is conservative and
  $\Zscr \subset \Dom{\widetilde{H}^*}$, then
  $\Xi^\TF_{\widetilde{H}^*}$ is a passive boundary node if and only if
  $(- L + \widetilde{H}^*)\rst{\Null{K}} = A_H^*$.
\end{enumerate}
\end{proposition}
\noindent If $\Xi = (G,L,K)$ is conservative, so is its time-flow
inverse $\Xi^\TF = (K, -L, G)$ by Definition~\ref{ConsNodeDef}. In
this case, it may be possible to use
Theorem~\ref{SecondConservativeMajorationThm} to conclude that
$\Xi^\TF_{\widetilde{H}^*}$ is a passive boundary node as well.  If both
$\Xi_H$ and $\Xi^\TF_{\widetilde{H}^*}$ are passive, then they cannot
be time-flow inverses of each other unless both nodes are, in fact,
conservative; i.e., $H = \widetilde{H}^* = 0$ on $\Zscr$.
\begin{proof}
It is easy to see that $A^* + T^* \subset (A + T)^*$ holds for
operators $A, T$ on $\Xscr$ with $\Dom{A} \cap \Dom{T}$ dense in
$\Xscr$ . Applying this on $A = L\rst{\Null{G}}$ and $T :=
\widetilde{H}\rst{\Null{G}}$ we get on $\Null{K}$ the inclusion $-L
\rst{\Null{K}} + \left ( \widetilde{H}\rst{\Null{G}} \right )^*
\subset A_H^*$. Here we used $A^* = -L \rst{\Null{K}}$ which holds
because $\Xi = (G,L,K)$ is a conservative boundary node whose dual
system (with semigroup generator $A^*$) coincides with the time-flow
inverse $\Xi^\TF = (K,-L,G)$.  Since $\Null{K} \subset
\Dom{\widetilde{H}^*}$ has been assumed, it follows that $\left
(\widetilde{H}\rst{\Null{G}} \right )^* z = \widetilde{H}^* z$ for all
$z \in \Null{K}$, and claim \eqref{AdjointSemigroupPropClaim1} now
follows.

The ``only if'' part of claims \eqref{AdjointSemigroupPropClaim2} and
\eqref{AdjointSemigroupPropClaim3}: By the internal well-posedness of
$\Xi^\TF_{\widetilde{H}^*}$, its main operator $(- L +
\widetilde{H}^*)\rst{\Null{K}}$ generates a $C_0$ semigroup, and its
resolvent set contains some right half plane by the Hille--Yoshida
theorem.  By claim \eqref{AdjointSemigroupPropClaim1} and the fact
that $A_H^*$ is (even maximally) dissipative, it follows that $(- L +
\widetilde{H}^*)\rst{\Null{K}}$ is dissipative. But then $(- L +
\widetilde{H}^*)\rst{\Null{K}}$ is maximally dissipative, and the
converse inclusion $A_H^* \subset (- L + \widetilde{H}^*)\rst{\Null{K}}$
follows.

The ``if'' part of claim \eqref{AdjointSemigroupPropClaim2}: The
operator $(- L + \widetilde{H}^*)\rst{\Null{K}}$ generates a
contraction semigroup on $\Xscr$ because it equals by assumption
$A_H^*$ where $A_H$ itself is a generator of a contraction semigroup
by Theorem~\ref{SecondConservativeMajorationThm}.

Equip the Hilbert space $\Dom{\widetilde{H}^*}$ with the graph norm of
the closed operator $\widetilde{H}^*$.  Since $\Zscr \subset
\Dom{\widetilde{H}^*}$ has been assumed, and both $\Zscr$ and
$\Dom{\widetilde{H}^*}$ are continuously embedded in $\Xscr$, the
inclusion $\Zscr \subset \Dom{\widetilde{H}^*}$ is continuous, too.
Now $\widetilde{H}^*\rst{\Zscr} \in \BLO(\Zscr;\Xscr)$ follows from
$\widetilde{H}^* \in \BLO(\Dom{\widetilde{H}^*};\Xscr)$.  Since now
$-L + \widetilde{H}^* \in \BLO(\Zscr;\Xscr)$, it follows that
$\Xi_{\widetilde{H}^*}^\TF$ is an internally well-posed boundary node
by \cite[Proposition~2.5]{M-S:CBCS}. (You could also argue by
verifying
Definition~\ref{BoundaryNodeDef}\eqref{BoundaryNodeDefRequirement2}
directly.)

The ``if'' part of claim \eqref{AdjointSemigroupPropClaim3}: The
``if'' part of claim \eqref{AdjointSemigroupPropClaim2} gives the
internal well-posedness of $\Xi^\TF_{\widetilde{H}^*}$. To show
passivity, only the Green--Lagrange inequality $ 2 \Re \ipd z{( -L +
  \widetilde{H}^*) z}_\Xscr \leq \norm{Kz}_\Yscr^2 -
\norm{Gz}_\Uscr^2$ is needed. This follows from \eqref{BackgroundG-L}
(by the conservativity of $\Xi^\TF$) and the dissipativity of
$\widetilde{H}^*$ with $\Zscr \subset \Dom {\widetilde{H}^*}$ (since
$\widetilde{H}$ is \emph{maximally} dissipative).
\end{proof}

\section{\label{WebsterSec} Generalised Webster's model for wave guides}

As proved in \cite{L-M:WECAD}, we arrive (under some mild
technical assumptions on $\Omega$ as explained in
\cite[Section~3]{L-M:WECAD}) to the following equations for the
approximate spatial averages of solutions of \eqref{WaveEq}:
\begin{equation} \label{WebstersEqBnrCtrl} 
  \begin{cases}
    & \psi_{tt} = \frac{c(s)^{2}}{A(s)} \frac{\partial}{\partial s}
    \left ( A(s) \frac{\partial \psi }{\partial s} \right ) 
- \frac{2 \pi \alpha W(s) c(s)^{2} }{A(s)} \frac{\partial \psi }{\partial t} \\
& 
    \text{\hspace{4.8cm}} \text{ for }  s \in (0,1) \text{ and } t \in \rplus, \\
    & - c(0) \psi_s(0,t) + \psi_t(0,t) = 2 \sqrt{\frac{c(0)}{\rho
        A(0)}} \, \tilde u(t) \quad
    \text{ for } t \in \rplus, \\
    & \psi(1,t) = 0 \quad  \text{ for } t \in \rplus, \quad \text{ and } \\
    & \psi(s,0) = \psi_0(s), \quad \rho \psi_t(s,0) = \pi_0(s) \quad
    \text{ for } s \in (0,1),
\end{cases}
\end{equation}
and the observation equation averages to
\begin{equation}\label{WebstersEqBnrCtrlObs}
  - c(0) \psi_s(0,t) - \psi_t(0,t) = 2 \sqrt{\frac{c(0)}{\rho
        A(0)}} \, \tilde y(t) \quad
    \text{ for } t \in \rplus.
\end{equation}
The notation has been introduced in Section~\ref{IntroSec}.
Analogously with the wave equation, the solution $\psi$ is called
\emph{Webster's velocity potential}.  In \cite[Section~3]{L-M:PEEWEWP}
we add a load function $f(s,t)$ to obtain the PDE $\psi_{tt} =
\frac{c(s)^{2}}{A(s)} \frac{\partial}{\partial s} \left ( A(s)
  \frac{\partial \psi }{\partial s} \right ) - \frac{2 \pi \alpha
  W(s)c(s)^2}{A(s)} \frac{\partial \psi }{\partial t} + f(s,t)$
because the argument there is based on the feed-forward connection
detailed in \cite[Fig.~1]{L-M:WECAD}. Only the boundary control input
is considered here, and it can be treated using boundary nodes.

We assume that the sound speed correction factor $\Sigma(s)$ and the
area function $A(s)$ are continuously differentiable for $s \in
[0,1]$, and that the estimates
\begin{equation} \label{UpperLowerEstimates}
  0 < \min_{s \in [0,1]}{A(s)} \leq \max_{s \in [0,1]}{A(s)} < \infty \text{ and }
  0 < \min_{s \in [0,1]}{c(s)} \leq \max_{s \in [0,1]}{c(s)} < \infty
\end{equation}
hold. These are natural assumptions recalling the geometry of the
tubular domain $\Omega$.  Define the operators
\begin{equation} \label{WebsterOpDef}
W := \frac{1}{A(s)} \frac{\partial}{\partial s} \left ( A(s) \frac{\partial}{\partial s}
\right ) \text{ and } D := - \frac{2 \pi  W(s)  }{A(s)}.
\end{equation}
The operator $D$ should be understood as a multiplication operator on
$L^2(0,1)$ by the strictly negative function $- 2 \pi W(\cdot)
A(\cdot)^{-1}$. Then the first of the equations in
\eqref{WebstersEqBnrCtrl} can be cast into first order form by using
the rule
\begin{equation*}
   \psi_{tt} = c(s)^{2} \left( W \psi + \alpha D \psi_t \right )  \quad \hat{=} \quad
  \frac{d}{dt} \bbm{\psi \\ \pi} 
  = \bbm{0 & \rho^{-1} \\  \rho c(s)^2 W & \alpha c(s)^2 D} \bbm{\psi \\ \pi}. 
\end{equation*}
Henceforth, let 
\begin{equation*}
L_W := \bbm{0 & \rho^{-1} \\ \rho c(s)^2 W & 0}:\Zscr_W \to \Xscr_W \text{ and }
H_W :=  \bbm{0 & 0 \\ 0 &  c(s)^2 D}:\Xscr_W \to \Xscr_W
\end{equation*}
where the Hilbert spaces are given by
\begin{align*}
  & \Zscr_W := \left (H_{\{1\}}\sp{1}(0,1) \cap H\sp{2}(0,1)\right ) \times
  H_{\{1\}}\sp{1}(0,1), \quad
  \Xscr_W := H_{\{1\}}\sp{1}(0,1) \times L\sp{2}(0,1) \\
  & \text{ where } \quad 
  H_{\{1\}}\sp{1}(0,1) := \left \{ f \in H\sp{1}(0,1): f(1) = 0 \right \}.
\end{align*}
Clearly we have $H_W \in \BLO(\Xscr_W)$, $H_W^* = H_W$, and this
operator is negative in the sense that $\ipd{ H_W \sbm{z_1 \\
    z_2}}{\sbm{z_1 \\ z_2} }_{\Xscr_W} = -2 \pi \int_{0}^1{
  \abs{z_2(s)}^2 W(s) c(s)^2 A(s)^{-1} \, ds} \leq 0$. So, the
operator $\alpha H_W$ for $\alpha > 0$ satisfies assumption
\eqref{SecondConservativeMajorationThmAss1} of
Theorem~\ref{SecondConservativeMajorationThm} with $b = 0$ and also
assumption \eqref{SecondConservativeMajorationThmAss2} of the same
theorem with $\tilde \Xscr = \Xscr$.

The Hilbert spaces $\Zscr_W$ and $\Xscr_W$ are equipped with the norms
\begin{align*}
  & \norm{\sbm{z_1 \\ z_2}}_{\Zscr_W}^2 := \norm{z_1}_{H^2(0,1)}^2 + \norm{z_2}_{H^1(0,1)}^2 
  \quad \text{ and } \\
  & \norm{\sbm{z_1 \\ z_2}}_{H\sp{1}(0,1) \times L\sp{2}(0,1)}^2 :=
  \norm{z_1}_{H\sp{1}(0,1)}^2 + \norm{z_2}_{L\sp{2}(0,1)}^2,
\end{align*}
respectively. We will use the \emph{energy norm} on $\Xscr_W$, which
for any $\rho > 0$ is defined by
\begin{equation} \label{WebstersEnergyNorm}
  \| \sbm{z_1 \\ z_2} \|_{\Xscr_W}^2 :=
  \frac{1}{2} \left (  \rho  \int_{0}^1{ \abs{z_1'(s)}^2A(s) \, ds} + 
    \frac{1}{\rho c^2} \int_{0}^1{ \abs{z_2(s)}^2 A(s) \Sigma(s)^{-2} \, ds} \right ).
\end{equation} 
This is an equivalent norm for $\Xscr_W$ because the conditions
\eqref{UpperLowerEstimates} hold and $\sqrt{2} \norm{z_1}_{L^2(0,1)}
\leq \norm{z_1'}_{L^2(0,1)}$ for all $z_1 \in
H_{\{1\}}\sp{1}(0,1)$. To see that the Poincar\'e inequality holds in
$H_{\{1\}}\sp{1}(0,1)$, note that for smooth functions $z$ with
$z(1)=0$, one has from the fundamental theorem of calculus that
\begin{equation*}
  \abs{z(s)}=\abs{\int_s^1z'(t)\,dt}\leq (1-s)^{1/2}\norm{z'}_{L^2(0,1)}.
\end{equation*}
From this, we proceed by squaring and integrating with respect to $s$,
and then passing to general Sobolev functions by approximation.
  
We define $\Uscr_W := \C$ with the absolute value norm $
\norm{u_0}_{\Uscr_W} := \abs{u_0}$.  The endpoint control and
observation functionals $G_W:\Zscr_W \to \Uscr_W$ and $K_W:\Zscr_W \to
\Uscr_W$ are defined by
\begin{align*}
  & G_W \sbm{z_1 \\ z_2} 
  := \frac{1}{2}\sqrt{\frac{A(0)}{ \rho  c(0)}} \left (- \rho c(0)z_1'(0) + z_2(0) \right ) \quad \text{ and } \\
  & K_W \sbm{z_1 \\ z_2} 
  := \frac{1}{2} \sqrt{\frac{A(0)}{\rho  c(0)}} \left (- \rho c(0)z_1'(0) - z_2(0) \right ).
\end{align*}
Now the generalised Webster's horn model
\eqref{WebstersEqBnrCtrl}--\eqref{WebstersEqBnrCtrlObs} for the state
$z(t) = \sbm{\psi(t) \\ \pi(t)}$ takes the form
\begin{align} \label{WebstersDiffEqNoF}
\begin{cases}
   &\frac{d}{dt}\sbm{\psi(t) \\ \pi(t)}  = \left ( L_W + \alpha H_W \right
  )\sbm{\psi(t)
    \\ \pi(t)},   \\
  & \tilde u(t) = G_W \sbm{\psi(t) \\ \pi(t)}, 
\end{cases}
\end{align}
and 
\begin{equation}
  \label{WebstersDiffEqNoFObs}
  \tilde y(t) = K_W \sbm{\psi(t) \\ \pi(t)}
\end{equation}
for all $t \in \rplus$. The initial conditions are $\sbm{\psi(0) \\
  \pi(0)} = \sbm{\psi_0 \\ \pi_0}$.  The state variable $\pi = \rho
\psi_t$ has the dimension of pressure, as for the wave equation.
 
The impedance passive version of the following
Theorem~\ref{WebsterNodeThm} is given in
\cite[Theorem~5.1]{A-M:CPBCS}, and it would be possible to deduce
parts of Theorem~\ref{WebsterNodeThm} from that result using the
external Cayley transform \cite[Definition 3.1]{M-S:IPCBCS}.  Here we
give a direct proof instead.

\begin{theorem} \label{WebsterNodeThm}
 Let the operators $L_W$, $H_W$, $G_W$, $K_W$, and spaces $\Zscr_W$,
 $\Xscr_W$, $\Uscr_W$ be defined as above.  Let $\sbm{\psi_0 \\ \pi_0}
 \in \Zscr_W$ and $\tilde u \in C^2(\rplus;\C)$ such that the
 compatibility condition $G_W\sbm{\psi_0 \\ \pi_0} = \tilde u(0)$
 holds.  Then for all $\alpha \geq 0$ the following holds:
  \begin{enumerate}
  \item \label{WebsterNodePropClaim1} The triple $\Xi^{(W)}_\alpha :=
    (G_W,L_W + \alpha H_W,K_W)$ is a scattering passive, strong
    boundary node on Hilbert spaces $(\Uscr_W,\Xscr_W,\Uscr_W)$.
    
  The semigroup generator $A_{W,\alpha} = \left (L_W + \alpha H_W \right
  )\rst{\Null{G_W}}$ of $\Xi^{(W)}_\alpha$ satisfies $A_{W,\alpha}^* =
  \left ( -L_W + \alpha H_W \right )\rst{\Null{K_W}} $ and $0 \in
  \rho(A_{W,\alpha}) \cap \rho(A_{W,\alpha}^*)$.
\item \label{WebsterNodePropClaim2} The equations in
  \eqref{WebstersDiffEqNoF} have a unique solution $\sbm{\psi \\ \pi}
  \in C^1(\rplus;\Xscr_W) \cap C(\rplus;\Zscr_W)$.  Hence we can
  define $\tilde y \in C(\rplus;\C)$ by equation
  \eqref{WebstersDiffEqNoFObs}.
  \item \label{WebsterNodePropClaim3} The solution of
    \eqref{WebstersDiffEqNoF} satisfies the energy dissipation inequality
    \begin{equation} \label{WebsterPassivityEstimate}
        \frac{d}{dt} \norm{\sbm{\psi(t) \\ \pi(t)}}_{\Xscr_W}^2  
        \leq \abs{\tilde u(t)}^2 - \abs{\tilde y(t)}^2 ,
 \quad t \in \rplus.
  \end{equation}
  \end{enumerate}
  Moreover, $\Xi^{(W)}_0$ is a conservative boundary node, and
  \eqref{WebsterPassivityEstimate} holds then as an equality.
\end{theorem}
\noindent Under the assumptions of this proposition, we have $\psi \in
C(\rplus; H^2(0,1)) \cap C^1(\rplus; H^1(0,1)) \cap C^2(\rplus;
L^2(0,1))$.
\begin{proof}
Claim \eqref{WebsterNodePropClaim1}:
  By Theorem~\ref{SecondConservativeMajorationThm}, it is enough to
  show the conservative case $\alpha = 0$.  Let us first verify the
  that the Green--Lagrange identity
  \begin{equation} \label{WebsterGLId}
    2 \Re \ipd {\sbm{z_1 \\ z_2}}{L_W \sbm{z_1 \\ z_2}}_{\Xscr_W} 
    + \abs{K_W \sbm{z_1 \\ z_2}}^2 = \abs{G_W \sbm{z_1 \\ z_2}}^2
  \end{equation}
  holds for all $\sbm{z_1 \\ z_2} \in \Zscr_W$. By partial
  integration,
  we get 
  \begin{equation*}
    2 \Re \ipd {\sbm{z_1 \\
        z_2}}{L_W \sbm{z_1 \\ z_2}}_{\Xscr_W} = - A(0) \Re \left (
      \overline{z_1'(0)} z_2(0) \right ).
  \end{equation*}
  Now \eqref{WebsterGLId}
  follows since $\abs{G_W \sbm{z_1 \\ z_2}}^2 - \abs{K_W \sbm{z_1 \\
      z_2}}^2 = - A(0) \Re \left ( \overline{z_1'(0)} z_2(0) \right )$
  just as in equations \eqref{nDimWaveConsPropEq2} --
  \eqref{nDimWaveConsPropEq3}.
  
  It is trivial that $G_W \Zscr_W = K_W \Zscr_W = \Uscr_W$ since
  $\dim{\Uscr_W} = 1$ and neither of the operators $G_W$ and $K_W$
  vanishes.  We prove next that $L_W$ maps $\Null{G_W}$ \emph{bijectively}
  onto $\Xscr_W$.  Now, $\sbm{z_1 \\ z_2} \in \Null{G_W}$ and
  $\sbm{w_1 \\ w_2} \in \Xscr_W$ satisfy $L_W \sbm{z_1 \\ z_2} =
  \sbm{w_1 \\ w_2}$ if and only if $z_2 = \rho w_1$ and
  \begin{equation*}
    \frac{\partial}{\partial s} \left (A(\cdot) \frac{\partial z_1}{\partial s}\right ) =  \frac{A(\cdot) w_2}{\rho c(\cdot)^{2} }  , \quad
    z_1(1) = 0, \quad z_1'(0) = \frac{w_1(0)}{c(0)} . 
  \end{equation*}
  Since this equation has always a unique solution $z_1 \in H^2(0,1)$
  for any $w_1 \in H_{\{1\}}\sp{1}(0,1)$ and $ w_2 \in L\sp{2}(0,1)$,
  it follows that $L_W \Null{G_W} = \Xscr_W$ and $0 \in \rho(A_{W,0})$
  where $A_{W,0} = L_W \rst{\Null{G_W}}$ is the semigroup generator of
  $\Xi^{(W)}_0$.  We conclude by
  Proposition~\ref{CheckingPassivityProp} that $\Xi^{(W)}_0$ is a
  conservative boundary node as claimed.  That $\Xi^{(W)}_\alpha$ is
  passive for $\alpha > 0$ with semigroup generator $A_{W,\alpha} =
  \left ( L_W + \alpha H_W \right ) \rst{\Null{G_W}}$ follows by
  Theorem~\ref{SecondConservativeMajorationThm}.

  Because $H_W^* = H_W \in \BLO(\Xscr)$ is dissipative, we may apply
  Theorem~\ref{SecondConservativeMajorationThm} again to the time-flow
  inverted, conservative node $\left ( \Xi^{(W) }_0 \right )^\TF =
  (K_W, -L_W, G_W)$ to conclude that the boundary node $(K_W, -L_W +
  \alpha H_W^*, G_W)$ is passive as well.  Claim
  \eqref{AdjointSemigroupPropClaim3} of
  Proposition~\ref{AdjointSemigroupProp} implies that $A_{W,\alpha}^*
  = \left ( - L_W + \alpha H_W \right )\rst{\Null{K_W}}$.

  Let us argue next that $0 \in \rho(A_{W,\alpha}) \cap
  \rho(A_{W,\alpha}^*)$ for $\alpha > 0$. Because $A_{W,\alpha}$ is a
  compact resolvent operator, it is enough to exclude $0 \in
  \sigma_p(A_{W,\alpha})$.  Suppose $A_{W,\alpha} z_0 = 0$, giving
  $\Re{\left < A_{W,0} z_0, z_0\right >_\Xscr} + \Re{\left < \alpha
    H_W z_0, z_0\right >_\Xscr} = \Re{\left < A_{W,\alpha} z_0, z_0
    \right >_\Xscr} = 0$. Thus $$\Re{\left < A_{W,0} z_0, z_0\right
    >_\Xscr} = \alpha \Re{\left < - H_W z_0 , z_0\right >_\Xscr} =
  \alpha \norm{(- H_W)^{1/2} z_0}^2_\Xscr = 0$$ by the dissipativity
  of both $A_{W,0}$ and $H_W$, and the fact that $- H_W$ is a
  self-adjoint nonnegative operator. Thus $z_0 \in \Null{H_W}$ and
  hence $A_{W,0} z_0 = (A_{W,0} + \alpha H_W) z_0 = A_{W,\alpha} z_0 =
  0$.  Because $0 \in \rho(A_{W,0})$ has already been shown, we
  conclude that $z_0 = 0$.

  The node $\Xi^{(W)}_0$ is strong (i.e., $L_W$ is closed with
    $\Dom{L_W} = \Zscr_W$) since $L_W = L_W^{**}$ and
$L_W^* = - L_W \rst{\Dom{L_W^*}}$ where 
\begin{equation*} 
  \Dom{L_W^*} = \left \{ \sbm{w_1 \\ w_2} \in H^1_{\{1\}}(0,1) \cap H^2(0,1)
  \times H^1_{0}(0,1) \, : \, \tfrac{\partial w_1}{\partial s}(0) = 0
 \right \}
\end{equation*}
which is dense in $\Xscr_W$ and satisfies $\Dom{L_W^*} \subset
\Dom{L_W}$.  That $\Xi^{(W)}_\alpha$ is strong for $\alpha > 0$ follows
from $H_W \in \BLO(\Xscr)$ as explained in
Theorem~\ref{SecondConservativeMajorationThm}.

  Claims \eqref{WebsterNodePropClaim2} and
  \eqref{WebsterNodePropClaim3} follow from
  Proposition~\ref{SolvabilityProp} and Eq.~\eqref{ConsCondDiff}.
\end{proof}

\section{\label{WaveSection} Passive wave equation on wave guides}

Define the tubular domain $\Omega \subset \R^3$ and its boundary
components $\Gamma$, $\Gamma(0)$, and $\Gamma(1)$ as in Section
\ref{IntroSec}.  Each of the sets $\Gamma$, $\Gamma(0)$, and
$\Gamma(1)$ are smooth manifolds but $\partial \Omega =
\closure{\Gamma} \cup \Gamma(0) \cup \Gamma(1)$ is only Lipschitz.
Other relevant properties of $\Omega$ and $\partial \Omega$ are listed
in \eqref{StandingAss1} -- \eqref{StandingAss3} of
Appendix~\ref{SoboAppendix} where we also make rigorous sense of the
Sobolev spaces, boundary trace mappings, Poincar\'e inequality, and
the Green's identity for such domains.

Following \cite[Section 3]{L-M:WECAD}, we consider the linear
dynamical system described by 
\begin{equation}\label{WaveEq}
\begin{cases}
  &  \phi_{tt}(\br, t) = c^2 \Delta \phi(\br, t) \quad
  \text{ for } \br \in \Omega \text{ and } t \in \rplus, \\
  &   c \frac{\partial \phi}{\partial
    \bnu}(\br, t) + \phi_t(\br, t) = 2 \sqrt{\tfrac{c}{\rho A(0)}} \, u(\br, t) \quad
  \text{ for } \br \in \Gamma(0) \text{ and }  t \in \rplus, \\
    &  \phi(\br, t)  = 0 \quad \text{ for } \br \in \Gamma(1) \text{ and }
  t \in \rplus, \\
  & \frac{\partial \phi}{\partial \bnu}(\br, t) + \alpha \phi_t(\br, t)  = 0 \quad \text{ for } \br
  \in \Gamma, \text{ and }
  t \in \rplus, \text{ and } \\
  & \phi(\br,0)  = \phi_0(\br), \quad \rho \phi_t(\br,0) = p_0(\br) \quad \text{ for
  } \br \in \Omega,
\end{cases}
\end{equation}
together with the observation $y$ defined by
\begin{equation}\label{WaveEqObs}
    c \frac{\partial \phi}{\partial
    \bnu} (\br, t) - \phi_t(\br, t)  =  2 \sqrt{\tfrac{c}{\rho A(0)}} \, y(\br, t)  \quad
  \text{ for }  \br \in \Gamma(0) \text{ and }  t \in \rplus.
\end{equation}
This model describes acoustics of a cavity $\Omega$ that has an open
end at $\Gamma(1)$ and an energy dissipating wall $\Gamma$. The
solution $\phi$ is the \emph{velocity potential} as its gradient is
the perturbation velocity field of the acoustic waves.  The boundary
control and observation on surface $\Gamma(0)$ (whose area is $A(0)$)
are both of scattering type. The speed of sound is denoted by $c > 0$.
The constants $\alpha \geq 0$ and $\rho > 0$ have physical meaning but
we refer to \cite{L-M:WECAD} for details.  Note that if $\alpha = 0$,
we have the Neumann boundary condition modelling a hard, sound
reflecting boundary on $\Gamma$.  Our purpose is to show that
\eqref{WaveEq}--\eqref{WaveEqObs} defines a passive boundary node
(conservative, if $\alpha = 0$ by a slightly different argument in
Corollary~\ref{nDimWaveConsThmCor}) by using
Theorem~\ref{FirstConservativeMajorationThm} with the aid of the
additional signals $\tilde u := \frac{1}{\sqrt{\alpha}} \frac{\partial
  \phi}{\partial \bnu} + \sqrt{\alpha} \phi_t$ (that will be grounded)
and $\tilde y := \frac{1}{\sqrt{\alpha}} \frac{\partial \phi}{\partial
  \bnu} - \sqrt{\alpha} \phi_t$ (that will be disregarded) on the wall
$\Gamma$.

The boundedness of the Dirichlet trace
implies that the space
\begin{equation} \label{HG0Def}
  H_{\Gamma(1)}^{1}(\Omega) := \left \{ f \in
    H^{1}(\Omega): f\rst{\Gamma(1)} = 0 \right \}.
\end{equation}
is a closed subspace of $H^{1}(\Omega)$. Define
\begin{equation} \label{Z0Def} 
  \tilde \Zscr' := \{ f \in H_{\Gamma(1)}^{1}(\Omega): \Delta f \in
  L^{2}(\Omega), \frac{\partial f}{ \partial \bnu}\rst{\Gamma(0) \cup
    \Gamma} \in L^{2}(\Gamma(0) \cup \Gamma) \}
\end{equation}
with the norm $ \| f \|_{\tilde \Zscr'}^{2} = \| f
\|_{H^{1}(\Omega)}^{2} + \| \Delta f \|_{L^{2}(\Omega)}^{2} +
\|\tfrac{\partial f}{ \partial \bnu}\rst{\Gamma(0) \cup \Gamma}
\|_{L^{2}(\Gamma(0) \cup \Gamma)}^{2}$.  Then the operator
\begin{equation}\label{NeumannTraceBndProp}
 \frac{\partial}{\partial \bnu}\rst{\Gamma'}:f \mapsto \frac{\partial
  f}{\partial \bnu}\rst{\Gamma'} \quad\text{lies in}\quad \BLO(\tilde \Zscr'
;L^{2}(\Gamma')) \quad\text{for}\quad
\Gamma' \in \{ \Gamma(0), \Gamma, \Gamma(0) \cup \Gamma \}.
\end{equation}

The spaces $\tilde \Zscr$, $\Xscr$, and the interior operator $L$ are
defined by
\begin{equation} \label{RestDefined}
\begin{aligned} 
  L := \sbm{0 & \rho^{-1} \\ \rho c^2 \Delta & 0}& :\tilde \Zscr \to \Xscr \quad  \text{ with } \\
  \tilde \Zscr := \tilde \Zscr' \times H_{\Gamma(1)}^{1}(\Omega) & \quad \text{ and } \quad 
   \Xscr := H_{\Gamma(1)}^{1}(\Omega) \times L^{2}(\Omega) 
\end{aligned}
\end{equation} 
where $H_{\Gamma(1)}^{1}(\Omega)$ and $\tilde \Zscr'$ are given by
\eqref{HG0Def} and \eqref{Z0Def}.  For the space $\Xscr$, we use the
\emph{energy norm}
\begin{equation} \label{EnergyNormEq} 
  \| \sbm{z_1 \\ z_2} \|_\Xscr^2 :=
  \frac{1}{2}\left ( \rho  \| |\nabla z_1| \|_{L^2(\Omega)}^2 
    + \frac{1}{\rho c^2}\|  z_2 \|_{L^2(\Omega)}^2 \right ).
\end{equation} 
The Poincar\'e inequality $\|z_1\|_{L^{2}(\Omega)} \leq M_{\Omega} \|
\nabla z_1 \|_{L^{2}(\Omega)}$ holds for $z_1 \in
H_{\Gamma(1)}^1(\Omega)$ as given in Theorem~\ref{PoincareInEqThm} in
Appendix~\ref{SoboAppendix}.  Therefore \eqref{EnergyNormEq} defines a
norm on $\Xscr$, equivalent to the Cartesian product norm
\begin{equation*}
  \| \sbm{z_1 \\ z_2} \|_{H^{1}(\Omega) \times L^{2}(\Omega)}^2 :=
  \|z_1\|_{L^{2}(\Omega)}^2  + 
  \| \nabla z_1 \|_{L^2(\Omega)}^2 + \|  z_2 \|_{L^2(\Omega)}^2
\end{equation*}
so that $\tilde \Zscr \subset \Xscr$ with a continuous embedding, and $L \in
\BLO(\tilde \Zscr;\Xscr)$ with respect to the $\tilde \Zscr$-norm
\begin{equation*}
  \| \sbm{z_1 \\ z_2} \|_{\tilde \Zscr}^2 :=
  \| z_1 \|_{\tilde \Zscr'}^2 + \|  z_2 \|_{L^2(\Omega)}^2 + 
  \| \nabla z_2 \|_{L^2(\Omega)}^2. 
\end{equation*} 
Defining $\Uscr := L^{2}(\Gamma(0))$ and $\tilde \Uscr :=
L^{2}(\Gamma)$ with the norms
\begin{equation}
  \label{WaveEqInputNorm}
  \norm{u_0}_{\Uscr}^2 = A(0)^{-1} \norm{u_0}_{L^{2}(\Gamma(0))}^2
  \text{ and } \norm{\tilde u_0}_{\tilde \Uscr} = \norm{\tilde
    u_0}_{L^{2}(\Gamma)},
\end{equation}
we get $\Uscr \oplus \tilde \Uscr = L^2(\Gamma(0) \cup \Gamma)$ where we use
the Cartesian product norm of $\Uscr$ and $\tilde \Uscr$.


The boundedness of the Dirichlet trace and the property
\eqref{NeumannTraceBndProp} of the Neumann trace imply that $\sbm{G \\
  G_\alpha} \in \BLO(\tilde \Zscr;\Uscr \oplus \tilde \Uscr)$ and
$\sbm{K \\ K_\alpha} \in \BLO(\tilde \Zscr;\Uscr \oplus \tilde \Uscr)$
where
\begin{equation}\label{GKOpDefinition}
  \begin{aligned}
  & \bbm{G \\ G_\alpha } \bbm{z_1 \\ z_2} := \frac{1}{2} \bbm{
    \sqrt{\frac{A(0)}{\rho c}} \left ( \rho c
    \frac{\partial z_1}{\partial \bnu}\rst{\Gamma(0)} +
    z_2\rst{\Gamma(0)} \right ) \\ 
\frac{\sqrt{\rho}}{\sqrt{\alpha}}
    \frac{\partial z_1}{\partial \bnu}\rst{\Gamma} + \frac{\sqrt{\alpha}}{\sqrt{\rho}}
    z_2\rst{\Gamma} } \quad \text{ and } \quad \\ & \bbm{K \\ K_\alpha
  } \bbm{z_1 \\ z_2} := \frac{1}{2} \bbm{ \sqrt{\frac{A(0)}{\rho c}}
    \left ( \rho c \frac{\partial z_1}{\partial \bnu}\rst{\Gamma(0)} -
    z_2\rst{\Gamma(0)} \right ) \\
 \frac{\sqrt{\rho}}{\sqrt{\alpha}}
    \frac{\partial z_1}{\partial \bnu}\rst{\Gamma} - \frac{\sqrt{\alpha}}{\sqrt{\rho}}
    z_2\rst{\Gamma}  } .
\end{aligned}
\end{equation}

The reason for defining the triple $\widetilde \Xi_\alpha := (\sbm{G
  \\ G_\alpha}, L, \sbm{K \\ K_\alpha})$ is to obtain first order
equations from \eqref{WaveEq}, using the equivalence of $\phi_{tt} =
c^2 \Delta \phi$ and $\frac{d}{dt} \sbm{\phi \\ p} = \sbm{0 &
  \rho^{-1} \\ \rho c^2 \Delta & 0} \sbm{\phi \\ p} $ where $p = \rho
\phi_t$ is the sound pressure. More precisely, equations
\eqref{WaveEq}--\eqref{WaveEqObs} are (at least formally) equivalent with
\begin{align} \label{ExtendedBdDiffEq}
  \begin{cases}
   & \frac{d}{dt} \bbm{\phi(t) \\ p(t)}  = L \bbm{\phi(t) \\ p(t)},  \\
   & \bbm{u(t) \\ 0}  = \bbm{G \\ G_\alpha}\bbm{\phi(t) \\ p(t)}, 
 \end{cases}
\end{align}
and
\begin{equation}
  \label{ExtendedBdDiffEqObs}
  \bbm{y(t) \\ \tilde y(t)}  = \bbm{ K \\ K_\alpha} \bbm{\phi(t) \\ p(t)}
\end{equation}
for $t \in \rplus$, with the initial conditions $\sbm{\phi(0) \\ p(0)}
= \sbm{\phi_0 \\ p_0}$.  The \emph{Green--Lagrange identity}
\begin{equation}\label{ConsBdIdent}
 2 \Re \ipd{\sbm{z_1 \\ z_2}}{L \sbm{z_1 \\ z_2} }_\Xscr + \norm{ \sbm{K \\ K_\alpha} \sbm{z_1 \\ z_2}}_{\Uscr
   \oplus \tilde \Uscr}^2 = \norm{\sbm{G \\ G_\alpha }\sbm{z_1 \\ z_2}}_{\Uscr \oplus
   \tilde \Uscr}^2 \text{ for all } \sbm{z_1 \\ z_2}
 \in \tilde \Zscr
\end{equation}
is a key fact for proving the conservativity of $\widetilde
\Xi_\alpha$, and we verify it next.  Green's identity
(Theorem~\ref{GreensIdentityThm} in Appendix~\ref{SoboAppendix}) gives
\begin{equation} \label{nDimWaveConsPropEq1} 
\begin{aligned} 
  & 2 \Re \left < \sbm{z_1 \\ z_2} , L \sbm{z_1 \\ z_2} \right >_\Xscr = 2 \Re \left < \sbm{z_1
    \\ z_2} , \sbm{ \rho^{-1} z_2 \\ \rho c^2 \Delta z_1} \right
  >_\Xscr \\ & = 2 \Re \frac{1}{2} \left ( \rho \int_{\Omega}{ \nabla
    \overline{z_1} \cdot \nabla (z_2/\rho) \, \dif V} + \frac{1}{\rho
    c^2} \left < \rho c^2 \Delta \overline{z_1},z_2 \right
  >_{L^2(\Omega)} \right ) \\ & = \Re \left ( \int_{\Gamma(0) \cup
    \Gamma \cup \Gamma(1)} \frac{\partial\overline{ z_1} }{\partial
    \bnu} z_2 \, \dif A \right ) \\ & = \Re \left < \frac{\partial
    z_1 } {\partial \bnu}\rst{\Gamma(0) }, z_2 \rst{\Gamma(0)} \right
  >_{L^2(\Gamma(0)) } + \Re \left < \frac{\partial z_1 } {\partial
    \bnu}\rst{\Gamma }, z_2 \rst{\Gamma} \right >_{L^2(\Gamma) }
\end{aligned}
\end{equation}
because $z_2\rst{\Gamma(1)} = 0$ by \eqref{RestDefined}.  On the other
hand, we obtain
    \begin{align} \label{nDimWaveConsPropEq2}
      & \norm{G \sbm{z_1 \\ z_2}}_{\Uscr}^2 = A(0)^{-1} \left < G \sbm{z_1 \\ z_2},
        G \sbm{z_1 \\ z_2} \right >_{L^2(\Gamma(0))} \\
      & = \frac{1}{4 \rho c} \left (\rho^2 c^2 \left \| \frac{\partial z_1 }
          {\partial \bnu}\rst{\Gamma(0)} \right \|_{L^2(\Gamma(0))}^2 +
        2 \rho c \, \Re \left < \frac{\partial z_1 } {\partial
            \bnu}\rst{\Gamma(0)}, z_2 \rst{\Gamma(0)} \right
        >_{L^2(\Gamma(0))} + \left \| z_2 \rst{\Gamma(0)} \right
        \|_{L^2(\Gamma(0))}^2 \right ) \nonumber
  \end{align}
  and also
    \begin{align} \label{nDimWaveConsPropEq3}
      & \norm{K \sbm{z_1 \\ z_2}}_{\Uscr}^2 = A(0)^{-1} \left < K \sbm{z_1 \\ z_2},
        K \sbm{z_1 \\ z_2} \right >_{L^2(\Gamma(0))} \\
      & = \frac{1}{4\rho c}  \left (\rho^2 c^2 \left \| \frac{\partial z_1 } {\partial
          \bnu}\rst{\Gamma(0)} \right \|_{L^2(\Gamma(0))}^2 
      -  2 \rho c \, \Re \left < \frac{\partial z_1 } {\partial
          \bnu}\rst{\Gamma(0)}, z_2 \rst{\Gamma(0)} \right >_{L^2(\Gamma(0))} 
      +  \left \| z_2 \rst{\Gamma(0)} \right \|_{L^2(\Gamma(0))}^2 \right ),
      \nonumber
  \end{align}
  where $G \sbm{z_1 \\ z_2}$ and $K \sbm{z_1 \\ z_2}$ are the first
  components in \eqref{GKOpDefinition} respectively.

Similarly, we compute the two terms needed in
    \begin{align} \label{nDimWaveConsPropEq4}
       & \norm{G_\alpha \sbm{z_1 \\ z_2}}_{\tilde \Uscr}^2 - \norm{K_\alpha \sbm{z_1 \\ z_2}}_{\tilde
         \Uscr}^2 \\ & = \left < G_\alpha \sbm{z_1 \\ z_2}, G_\alpha
       \sbm{z_1 \\ z_2} \right >_{L^2(\Gamma)} - \left < K_\alpha
       \sbm{z_1 \\ z_2}, K_\alpha \sbm{z_1 \\ z_2} \right
       >_{L^2(\Gamma)}  = \Re \left < \frac{\partial z_1 }
       {\partial \bnu}\rst{\Gamma}, z_2 \rst{\Gamma} \right
       >_{L^2(\Gamma )}, \nonumber
  \end{align}
  where $G_\alpha \sbm{z_1 \\ z_2}$ and $K_\alpha \sbm{z_1 \\ z_2}$
  are the second components in \eqref{GKOpDefinition} respectively.
  Now \eqref{nDimWaveConsPropEq1} -- \eqref{nDimWaveConsPropEq4}
  implies \eqref{ConsBdIdent} as required.
  
  We proceed to show that the the triple $\Xi_\alpha :=
  (G\rst{\Zscr_{\alpha}},L\rst{\Zscr_{\alpha}},K\rst{\Zscr_{\alpha}})$
  for all $\alpha > 0$ is a scattering passive boundary node on
  Hilbert spaces $(\Uscr,\Xscr,\Uscr)$ with the solution space
  \begin{equation} \label{FinalZ0Def} 
    \Zscr_{\alpha} := \left \{ \bbm{z_1 \\ z_2 } \in \tilde \Zscr' \times H_{\Gamma(1)}^{1}(\Omega):  
    \frac{\partial z_1}{ \partial \bnu}\rst{\Gamma} + \frac{\alpha}{\rho} z_2 \rst{\Gamma} = 0 \right \}.
  \end{equation}
  Note that $\Zscr_{\alpha}$ is a closed subspace of $\tilde \Zscr$
  because $G_\alpha \in \BLO(\tilde \Zscr;\tilde \Uscr)$ and $\Zscr_{\alpha} =
  \Null{G_\alpha}$. Therefore, we can use the norm of $\tilde \Zscr$
  on $\Zscr_{\alpha}$.  The conservative case $\alpha = 0$ is slightly
  different, and it is treated separately in
  Corollary~\ref{nDimWaveConsThmCor}.
\begin{theorem}
  \label{nDimWaveConsThm}
  Take $\alpha > 0$ and let the operators $L$, $G$, $K$, and Hilbert
  spaces $\Xscr$, $\Uscr$, and $\Zscr_{\alpha}$ be defined as above.
  Let $\sbm{\phi_0 \\ p_0} \in \Zscr_{\alpha}$ and $u \in
  C^2(\rplus;\Uscr)$ such that the compatibility condition
  $G\sbm{\phi_0 \\ p_0} = u(0)$ holds. Then the following holds:
  \begin{enumerate}
  \item \label{WaveNodePropClaim1} The triple $\Xi_\alpha :=
    (G\rst{\Zscr_{\alpha}},L\rst{\Zscr_{\alpha}},K\rst{\Zscr_{\alpha}})$
    is a scattering passive boundary node on Hilbert spaces
    $(\Uscr,\Xscr,\Uscr)$ with solution space $\Zscr_{\alpha}$.  The
    semigroup generator $A_\alpha = L\rst{\Null{G} \cap \Null{
        G_\alpha}}$ of $\Xi_\alpha$ satisfies $A_\alpha^* =
    -L\rst{\Null{K} \cap \Null{K_\alpha}}$ and $0 \in \rho(A_\alpha)
    \cap \rho(A_\alpha^*)$.
  \item \label{WaveNodePropClaim2} The equations\footnote{Note that
    \eqref{BckGrndDynEq} is equivalent with \eqref{WaveEq} and
    \eqref{ExtendedBdDiffEq} in the context of this theorem. }  in
    \eqref{ExtendedBdDiffEq} have a unique solution $\sbm{\phi \\ p }
    \in C^1(\rplus;\Xscr) \cap C(\rplus;\Zscr_{\alpha})$.  Hence we
    can define $y \in C(\rplus;\Uscr)$ by equation
    \eqref{ExtendedBdDiffEqObs}.
  \item \label{WaveNodePropClaim3} The solution of
    \eqref{ExtendedBdDiffEq} satisfies the energy dissipation
    inequality
    \begin{equation} \label{WaveEnergyBalance}
      \frac{d}{dt}\norm{\sbm{\phi(t) \\ p(t)}}_\Xscr^2  
      \leq \norm{u(t)}_\Uscr^2 - \norm{y(t)}_\Uscr^2,
 \quad t \in \rplus.
  \end{equation}
\end{enumerate}
\end{theorem}
\noindent It follows from claim \eqref{WaveNodePropClaim2} and the
definition of the norms of $\Zscr_{\alpha}$ and $\Xscr$ that $\phi \in
C^1(\rplus; H^1(\Omega)) \cap C^2(\rplus; L^2(\Omega))$, $\nabla \phi
\in C^1(\rplus;L^2(\Omega;\R^3))$, and $\Delta \phi \in
C(\rplus;L^2(\Omega))$. These are the same smoothness properties that
have been used in \cite[see, in particular, Eq.~(1.4)]{L-M:WECAD} for
deriving the generalised Webster's equation in \eqref{IntroWebstersEq}
from the wave equation.
\begin{proof}
  Claim \eqref{WaveNodePropClaim1}: By
  Theorem~\ref{FirstConservativeMajorationThm} and the discussion
  preceding this theorem, it is enough to show that $\widetilde
  \Xi_\alpha = (\sbm{G \\ G_\alpha}, L, \sbm{K \\ K_\alpha})$
  introduced above is a conservative boundary node which is easiest
  done by using Proposition~\ref{CheckingPassivityProp}.  Since the
  Green--Lagrange identity \eqref{BackgroundG-L} has already been
  established, it remains to prove conditions \eqref{PassCharCond2}
  (with $\sbm{G \\ G_\alpha}$ in place of $G$) and
  \eqref{PassCharCond4} (with $\sbm{K \\ K_\alpha}$ in place of $K$)
  of Proposition~\ref{CheckingPassivityProp} with $\beta = \gamma =
  0$.  It is enough to consider only $\beta = \gamma = 0$ because the
  resolvent sets of $L\rst{\Null{G}}$ and $-L\rst{\Null{K}}$ in
  Proposition~\ref{CheckingPassivityProp} are open, and then the same
  conditions hold for some $\beta,\gamma > 0$ as well.
  
  For an arbitrary $g \in L^2(\Gamma(0)\cup \Gamma)$ there exists a
  unique variational\footnote{We leave it to the interested reader to
    derive the variational form using Green's identity
    \eqref{GreenIdentity} and then carry out the usual argument by the
    Lax--Milgram theorem; see, e.g., \cite[Lemma~2.2.1.1]{PG:EPNSD}.} solution $z_1 \in H^1_{\Gamma(1)}(\Omega)$
  of the problem
   \begin{equation}\label{UniqueSolvLaplacian}
     \Delta z_1 = 0, \quad z_1\rst{ \Gamma(1)} = 0,  \quad
     \frac{\partial z_1}{ \partial \nu}\rst{\Gamma(0) \cup \Gamma} = g.
   \end{equation}
   Since $z_1 \in \tilde \Zscr'$, we have $\tfrac{\partial}{\partial
     \nu}\rst{\Gamma(0) \cup \Gamma} \tilde \Zscr' = L^{2}(\Gamma(0)
   \cup \Gamma)$ which obviously gives both $\tfrac{\partial}{\partial
     \nu}\rst{\Gamma(0)} \tilde \Zscr' = L^{2}(\Gamma(0))$ and
   $\tfrac{\partial}{\partial \nu}\rst{\Gamma} \tilde\Zscr' =
   L^{2}(\Gamma)$.  Clearly $\tilde\Zscr' \oplus \{ 0 \} \subset
   \tilde \Zscr$ and the surjectivity of $\sbm{G\\ G_\alpha}$ follows
   from $$\bbm{G \\ G_\alpha } \bbm{z_1 \\ 0} := \frac{1}{2} \bbm{
     \sqrt{A(0)\rho c}
     \frac{\partial }{\partial \bnu} \rst{\Gamma(0)} \\
     \frac{\sqrt{\rho}}{\sqrt{\alpha}} \frac{\partial }{\partial
       \bnu}\rst{\Gamma}} z_1.$$ To see this, for a given $h\in
   L^2(\Gamma(0)\cup \Gamma)$, we choose
   \begin{equation*}
     g=
     \begin{cases}
       2\frac{1}{\sqrt{A(0)\rho c}}h, & \text{on }\Gamma(0),\\
       2\frac{\sqrt{\alpha}}{\sqrt{\rho}} h, &\text{on }\Gamma
     \end{cases}
   \end{equation*}
   in \eqref{UniqueSolvLaplacian} to find a function $z_1$ so that
   $\sbm{G \\ G_\alpha } \sbm{z_1 \\ 0}=h$. The surjectivity of $\sbm{K \\
     K_\alpha}$ is proved similarly.

   To show that $L \Null{\sbm{G\\ G_\alpha}} = L \left (
     \Null{G} \cap \Null{G_\alpha} \right ) = \Xscr$,  let $\sbm{w_1 \\
     w_2} \in \Xscr$ be arbitrary.  Then $\sbm{w_1 \\ w_2}
   = L \sbm{z_1 \\ z_2} = \sbm{\rho^{-1} z_2 \\ \rho c^2 \Delta z_1}$ for $\sbm{z_1 \\
     z_2} \in \Null{G} \cap \Null{G_\alpha}$ if and only if $z_2 =
   \rho w_1$ and the variational solution $z_1 \in
   H_{\Gamma(1)}^{1}(\Omega)$ of the problem
   \begin{equation*}
     \rho c^2 \Delta z_1 = w_2, \quad
     z_1\rst{\Gamma(1)} = 0, \quad
     \frac{\partial z_1}{ \partial \nu}\rst{\Gamma} = - \alpha \rho w_1 \rst{\Gamma},  \quad
     c \frac{\partial z_1}{ \partial \nu}\rst{\Gamma(0)} 
     = -  w_1 \rst{\Gamma(0)}
   \end{equation*} 
   exists and belongs to the space $\Zscr'$. Now, this condition can
   be verified by standard variational techniques because $w_2 \in
   L^2(\Omega)$ and $w_1 \in H_{\Gamma(1)}^{1}(\Omega)$ which implies
   $w_1 \rst{\Gamma(0) \cup \Gamma} \in H^{1/2}(\Gamma(0) \cup \Gamma)
   \subset L^2(\Gamma(0) \cup \Gamma)$. That $L \Null{\sbm{K
       \\ K_\alpha}} = \Xscr$ is proved similarly.  All the conditions
   of Proposition~\ref{CheckingPassivityProp} are now satisfied with
   $\beta = \gamma = 0$, and thus $\tilde \Xi_\alpha$ is a
   conservative boundary node. It now follows from
   Theorem~\ref{FirstConservativeMajorationThm} that $\Xi_\alpha$ is a
   passive boundary node which has the common semigroup generator
   $A_\alpha = L\rst{\Null{G} \cap \Null{G_\alpha}}$ with the original
   conservative boundary node $\tilde \Xi_\alpha$.  By \cite[Theorem
     1.9 and Proposition~4.3]{M-S:CBCS}, the dual system of
   $\widetilde \Xi_\alpha$ is of boundary control type, and it
   coincides with the time-flow inverted boundary node $\widetilde
   \Xi_\alpha^\TF$. Now, the unbounded adjoint $A_\alpha^*$ is the
   semigroup generator of the dual system $\widetilde \Xi_\alpha^\TF$,
   and hence $A_\alpha^* = -L\rst{\Null{K} \cap \Null{K_\alpha}}$ as
   claimed.

   It remains to show that $0 \notin \sigma(A_\alpha)$. We have
   already shown above that $A_\alpha \Dom{A_\alpha} = \Xscr$ with
   $\Dom{A_\alpha} = \Null{G} \cap \Null{G_\alpha}$, and the remaining
   injectivity part follows if we show that $\Null{L} \cap \Null{G}
   \cap \Null{G_\alpha} = \{ 0 \}$. This follows because the variational
solution in $H^1(\Omega)$ of the homogenous problem
  \begin{equation*}
     \Delta z_1 = 0, \quad
     z_1\rst{\Gamma(1)} = 0, \quad
     \frac{\partial z_1}{ \partial \nu}\rst{\Gamma(0) \cup \Gamma} = 0  
   \end{equation*} 
  is unique.  That $0 \notin \sigma(A_\alpha^*)$ follows similarly by
  considering the time-flow inverted system $\widetilde \Xi_\alpha^\TF$
  instead.

 Claims \eqref{WaveNodePropClaim2} and \eqref{WaveNodePropClaim3}:
 Since scattering passive boundary nodes are internally well-posed, it
 follows from, e.g., \cite[Lemma 2.6]{M-S:CBCS} that equations
 \eqref{BckGrndDynEq} are solvable as has been explained in
 Section~\ref{BackgroundSection}.
\end{proof}
\begin{corollary}
  \label{nDimWaveConsThmCor}
  Use the same notation and make the same assumptions as in
  Theorem~\ref{nDimWaveConsThm}.  If $\alpha = 0$, then claims
  \eqref{WaveNodePropClaim1} --- \eqref{WaveNodePropClaim3} of
  Theorem~\ref{nDimWaveConsThm} hold in the stronger form: {\rm (i')}
  the triple $\Xi_0 := (G\rst{\Zscr_0},L\rst{\Zscr_0},K\rst{\Zscr_0})$
  is a scattering conservative boundary node on Hilbert spaces
  $(\Uscr,\Xscr,\Uscr)$ with the solution space $\Zscr_0 := \tilde
  \Zscr'_0 \times H_{\Gamma(1)}^{1}(\Omega)$ where
\begin{equation} \label{Z00Def} 
  \tilde \Zscr'_0 := \{ f \in H_{\Gamma(1)}^{1}(\Omega): \Delta f \in
  L^{2}(\Omega), \frac{\partial f}{ \partial \bnu}\rst{\Gamma(0) } \in L^{2}(\Gamma(0)),  \frac{\partial f}{ \partial \bnu}\rst{\Gamma } = 0\};
\end{equation}
and {\rm (iii')} the energy inequality \eqref{WaveEnergyBalance} holds
as an equality.  
\end{corollary}
\noindent 
Claim \eqref{WaveNodePropClaim2} of Theorem~\ref{nDimWaveConsThm}
remains true without change.  Thus, the solution $\phi$ has the same
regularity properties as listed right after
Theorem~\ref{nDimWaveConsThm}.
\begin{proof}
  Because the operators $G_\alpha$ and $K_\alpha$ refer to
  $1/\sqrt{\alpha}$, we cannot simply set $\alpha = 0$ in the proof.
  This problem could be resolved by making the norm of $\tilde \Uscr$
  dependent on $\alpha$ which we want to avoid. A direct argument can
  be given without ever defining $\widetilde \Xi_\alpha$. To prove the
  Green--Lagrange identity
  \begin{equation}\label{ConservativeConsBdIdent}
    2 \Re \ipd{\sbm{z_1 \\ z_2}}{L \sbm{z_1 \\ z_2} }_\Xscr + \norm{ K \sbm{z_1 \\ z_2}}_{\Uscr}^2 
    = \norm{G \sbm{z_1 \\ z_2}}_{\Uscr}^2 \text{ for all } \sbm{z_1 \\ z_2}
    \in \tilde \Zscr_0
  \end{equation}
  for $\Xi_0$, one simply omits the last term on the right hand side
  of \eqref{nDimWaveConsPropEq1} by using the Neumann condition
  $\frac{\partial{ z_1} }{\partial \bnu} \rst{\Gamma} = 0$ from
  \eqref{Z00Def}.  Then \eqref{ConservativeConsBdIdent} follows from
  \eqref{nDimWaveConsPropEq1}---\eqref{nDimWaveConsPropEq3}, leading
  ultimately to \eqref{WaveEnergyBalance} with an equality. The
  remaining parts of claim {\rm (i')} follow by the argument given in
  the proof of Theorem~\ref{nDimWaveConsThm}.
\end{proof}
This result generalises the reflecting mirror example in
\cite[Section~5]{M-S:CBCS}, and further generalisations are given in
Section~\ref{ConclSec}.

\section{\label{ConclSec} Conclusions and generalisations}

We have given a unified treatment of a 3D wave equation model on
tubular structures and the corresponding Webster's horn model in the
form it is derived and used in \cite{L-M:PEEWEWP,L-M:WECAD}. Both the
forward time solvability and the energy inequalities have been treated
rigorously, and the necessary but hard-to-find Sobolev space apparatus
was presented in App.~\ref{SoboAppendix}.  The strictly dissipative
case was reduced to the conservative case using auxiliary
Theorems~\ref{FirstConservativeMajorationThm} and
\ref{SecondConservativeMajorationThm} that have independent interest.
  
Theorem~\ref{nDimWaveConsThm} can be extended and generalised
significantly using only the techniques presented in this work.
Firstly, a dissipation term, analogous with the one appearing in
Webster's equation \eqref{WebstersEqBnrCtrl}, can be added to the wave
equation part of \eqref{WaveEq} while keeping rest of the model the
same:
\begin{corollary}
\label{nDimWaveConsThmCor1}
  Theorem~\ref{nDimWaveConsThm} remains true if the wave equation
  $\phi_{tt} = c^2 \Delta \phi$ in \eqref{WaveEq} is replaced by
  $\phi_{tt} = c^2 \Delta \phi + g(\cdot) \phi_t$ where $g$ is a
  smooth function satisfying $g(\br) \leq 0$ for all $\br \in \Omega$.
\end{corollary}
\noindent Indeed, this follows by using
Theorem~\ref{SecondConservativeMajorationThm} on the result of
Theorem~\ref{nDimWaveConsThm} in the same way as has been done in
Section~\ref{WebsterSec}. Even now the resulting negative perturbation
$H$ on the original interior operator $L$ in \eqref{RestDefined}
satisfies $H \in \BLO(\Xscr)$. The same dissipation term can, of
course, be added to Corollary~\ref{nDimWaveConsThmCor} (where $\alpha =
0$) as well but then the resulting boundary node is only passive
unless $g \equiv 0$.

Theorem~\ref{nDimWaveConsThm} can be generalised to cover much more
complicated geometries $\Omega \subset \R^3$ than tube segments with
circular cross-sections. Inspecting the construction of the boundary
node $\Xi_\alpha$ and the accompanying Hilbert spaces in
Section~\ref{WaveSection}, it becomes clear that much more can be
proved at the cost of more complicated notation but nothing more:
\begin{corollary}
\label{nDimWaveConsThmCor2}
Let $\Omega \subset \R^3$ be a bounded Lipschitz domain satisfying
standing assumptions \eqref{StandingAss1} -- \eqref{StandingAss4} in
App.~\ref{SoboAppendix}. Denote the smooth boundary components of
$\Omega$ by $\Gamma_j$ where $j \in J \subset \N$ satisfying $\partial
\Omega = \cup_{j \in J}{\overline{\Gamma_j}}$. Let $J = J_1 \cup J_2
\cup J_3$ where the sets are pairwise disjoint, and at least $J_1$ and
$J_3$ are nonempty.  Define the open Lipschitz surfaces $\Gamma(0),
\Gamma, \Gamma(1) \subset \partial \Omega$ through their closures
$\overline{\Gamma(0)} = \ \cup_{j \in J_1}{\overline{\Gamma_j}}$,
$\overline{\Gamma} = \ \cup_{j \in J_2}{\overline{\Gamma_j}}$, and
$\overline{\Gamma(1)} = \ \cup_{j \in J_3}{\overline{\Gamma_j}}$,
respectively. Let $\alpha = \{ \alpha_j \}_{j \in J_2} \subset
(-\infty, 0]$ be a vector of dissipation parameters.

Then the wave equation model \eqref{WaveEq} with equations
\begin{equation*}
  \alpha_j \frac{\partial \phi}{\partial t}(\br, t) + \frac{\partial
    \phi}{\partial \bnu}(\br, t) = 0 \quad \text{ for all } \br \in
  \Gamma_j,   \, \, t \geq 0,  \text{ and } j \in J_2
\end{equation*}
in place of the fourth equation in \eqref{WaveEq} defines the boundary
node $\Xi_{\alpha}$ and the Hilbert spaces $\Xscr$, $\Uscr$, and
$\Zscr_{\alpha}$ in a same way as presented in
Section~\ref{WaveSection}. Moreover, Theorem~\ref{nDimWaveConsThm} and
Corollary~\ref{nDimWaveConsThmCor} (where $\alpha_j = 0$ for all $j
\in J_2$) hold without change.
\end{corollary}
\noindent In particular, the set $\Omega$ may be an union of a finite
number of tubular domains described in Section~\ref{IntroSec}. Even
loops are possible and the interior domain dissipation can be added
just like in Corollary~\ref{nDimWaveConsThmCor1}. This configuration
can be found in the study of the spectral limit behaviour of
Neumann--Laplacian on graph-like structures in
\cite{K-Z:CSMSCG,R-S:VPMVTSI}.
\begin{proof}[Comments on the proof]
The argument in Section~\ref{WaveSection} defines $\Xi_\alpha$, the
Hilbert spaces $\Xscr$, $\Uscr$, and $\Zscr_{\alpha}$, and the
Green--Lagrange identity by splitting $\partial \Omega$ into three
smooth components and patching things up using the results of
App.~\ref{SoboAppendix}. The same can be done on any finite number of
components since the results of App.~\ref{SoboAppendix} are
sufficiently general to allow it.  The solvability of the variational
problems in the proof of Theorem~\ref{nDimWaveConsThm} do not depend
on the number of such boundary components either.
\end{proof}
There is nothing in Section~\ref{WaveSection} that would exclude the
further generalisation to $\Omega \subset \R^n$ for any $n \geq 2$ if
standing assumptions \eqref{StandingAss1} -- \eqref{StandingAss4} in
App.~\ref{SoboAppendix} remain true.  If $n = 2$ and $\Omega$ is a
curvilinear polygon (i.e., it is simply connected), the necessary PDE
toolkit can be found in \cite[Section~1]{PG:EPNSD}.

Also Theorem~\ref{WebsterNodeThm} has extensions but not as many as
Theorem~\ref{nDimWaveConsThm}. Firstly, the nonnegative constant
$\alpha$ can be replaced by a nonnegative function $\alpha(\cdot) \in
C[0,1]$ since the $s$-dependency is already present in the operator
$D$ in \eqref{WebsterOpDef}. Secondly, strong boundary nodes described
by Theorem~\ref{WebsterNodeThm} can be scaled to different interval
lengths and coupled to finite \emph{transmission graphs} as explained
in \cite{A-M:CPBCS} for impedance passive component systems.  The full
treatment of a simple transmission graph, consisting of three
Webster's horn models in Y-configuration, has been given in
\cite[Theorem~5.2]{A-M:CPBCS}. More general finite configurations can
be treated similarly, and the resulting impedance passive system can
be translated to a scattering passive system by the external Cayley
transform \cite[Section~3]{M-S:IPCBCS}, thus producing a
generalisation of Theorem~\ref{WebsterNodeThm}. We note that there is
not much point in trying to derive the transmission graph directly
from scattering passive systems since the continuity equation (for the
pressure) and Kirchhoff's law (for the conservation of flow) at each
node is easiest described by impedance notions.


That Theorem~\ref{SecondConservativeMajorationThm} cannot be used for
all possible dissipation terms is seen by considering the wave
equation with Kelvin--Voigt structural damping term
\begin{equation}
  \psi_{tt} = c^2 \psi_{ss} +  \frac{\partial }{\partial s} \left
  (\beta(s) \frac{\partial }{\partial s} \psi_{t} \right )
\quad \text{ where } \quad \beta(s) \geq 0.
\end{equation}
For details of this dissipation model, see, e.g., \cite{L-L:EDEVSLV}.
To obtain the full dynamical system analogous to the one associated
with Webster's equation, the same boundary and initial conditions can
be used as in \eqref{IntroWebstersEq} for $\beta \in C^\infty[0,1]$
compactly supported $(0,1)$. Thus the operators $G_W$ and $K_W$ do not
change. Following Section~\ref{WebsterSec} we use the velocity
potential and the pressure as state variables $\sbm{\psi \\ \pi}$.  We
define the Hilbert spaces $\Zscr_W$ and $\Xscr_W$ similarly as well as
the operators
\begin{equation*}
\begin{aligned}
L_W & := \bbm{0 & \rho^{-1} \\ \rho c^2 \frac{\partial^2}{\partial s^2}
  & 0}:\Zscr_W \to \Xscr_W \text{ and } \\ \widetilde{H} & := \bbm{0 & 0
  \\ 0 & \frac{\partial }{\partial s} \left (\beta(s) \frac{\partial
  }{\partial s} \right ) }:\Dom{\widetilde{H}}\subset \Xscr_W \to \Xscr_W
\end{aligned}
\end{equation*}
where $\Dom{\widetilde{H}} := H^1_{\{ 1 \}}(0,1) \times \{f \in
L^2(0,1) : \beta(s) \tfrac{\partial f}{\partial s} \in H^1(0,1) \}$.
The physical energy norm for $\Xscr_W$ is given by
\eqref{WebstersEnergyNorm} with $A(s) = \Sigma(s) \equiv 1$
representing a constant diameter straight tube. If the parameter
$\beta \equiv 0$, the colligation $(G_W, L_W, K_W )$ is a special case
of the conservative system $\Xi_0^{(W)}$ described in
Theorem~\ref{WebsterNodeThm}. Clearly, the domain of $\widetilde{H}$
cannot be further extended without violating the range inclusion in
$\Xscr_W$. On the other hand, the inclusion $\Zscr \subset
\Dom{\widetilde{H}}$ required by
Theorem~\ref{SecondConservativeMajorationThm} is not satisfied.

\subsection*{Acknowledgment}
\noindent The authors have received support from the Finnish Graduate School on
 Engineering Mechanics, the Norwegian Research Council, and Aalto
 Starting Grant (grant no. 915587).  The authors wish to thank the
 anonymous referees for many valuable comments.


\appendix

\section{\label{SoboAppendix} Sobolev spaces and Green's identity}

We prove a sufficiently general form of Green's identity that holds in
a tubular domain $\Omega$ (that has a Lipschitz boundary) with
minimal assumptions on any functions involved. We make the following
standing assumptions on $\Omega$:
  \begin{enumerate}
  \item \label{StandingAss1} $\Omega$ is a bounded Lipschitz domain so
    that $\Omega$ locally on one side of is boundary $\partial
    \Omega$;
  \item \label{StandingAss2} there is a finite number of smooth, open,
    connected, and disjoint $(n-1)$-dimensional surfaces $\Gamma_j$
    with the following property: the boundary $\partial \Omega$ is a
    union of all $\Gamma_j$'s and parts of their common boundaries
    $\overline{\Gamma}_j\cap\overline{\Gamma}_k$ for $j\not=k$;
  \item \label{StandingAss3} $\Ha^{n-2}(\overline{\Gamma}_j\cap\overline{\Gamma}_k)<\infty
    $ for all $j\not=k$ where $\Ha^{m}(M)$ is the $m$-dimensional
    Hausdorff measure for $1 \leq m \leq n$ of $M \subset \R^n$;  and
  \item \label{StandingAss4} for each $j$, there is a $C^{\infty}$ vector field $\bnu_j$
    defined in a neighbourhood of $\overline{\Omega}$ such that
    $\bnu_j(\br)$ is the exterior unit normal to $\Gamma_j$ at $\br
    \in \Gamma_j$.
  \end{enumerate}  
  That $\Gamma_j \subset \R^n$ is an open, bounded, and smooth
  $(n-1)$-dimensional surface means plainly the following: there is an
  open and bounded $\tilde \Gamma_j \subset \R^{ n- 1}$ and a
  $C^{\infty}$-diffeomorphism $\phi_j$ from $\tilde \Gamma_j$ onto
  $\Gamma_j$. The pair $(\phi_j, \tilde \Gamma_j)$ is a global
  coordinate representation of $\Gamma_j$.



The boundary conditions in Section \ref{WaveSection} involve Dirichlet
conditions on some parts of the boundary $\partial \Omega$ and Neumann
type conditions on other parts of the same \emph{connected component}
of $\partial \Omega$. All this is in contrast with the inconvenient
technical assumption on $\partial \Omega$ in, e.g.,
\cite{JL:DSWEBRBD,M-S:CBCS,RT:WEBDBDOA} that must be avoided in the
verification of the Green--Lagrange identity in Section
\ref{WaveSection} and elsewhere.  We need a version of Green's
identity suitable for this situation. This is in Theorem
\ref{GreensIdentityThm} below. The key fact ensuring the validity of
this identity is that the interfaces where we switch between different
boundary conditions are so small that Sobolev functions do not see
them. That this is the case is a consequence of the assumption
\eqref{StandingAss3} above, and it is expressed rigorously in the
following auxiliary result.
\begin{lemma}\label{lem:zerocap}
  Let $\Omega$ be a bounded domain with a Lipschitz boundary, and let
  $E \subset \R^n$ be a compact set of zero capacity; i.e.,
  \begin{equation}
    \label{CapacityDef}
    C(E) := \inf_{u \in S(E)}\int_{\R^n}{\left ( \abs{u}^2+\abs{\nabla u}^2 \right ) \dif V} = 0
  \end{equation}
  where 
  \begin{equation*}
    S(E) := \{u \in C^\infty(\R^n) : 0\leq u\leq 1 \text{ in } \R^n \text{ and } 
    u = 1 \text{ in } N, \text{ where } N \text{ is open and } E \subset N   \}.
  \end{equation*}
  Then
  \begin{enumerate} 
  \item \label{lem:zerocapClaim1}
    the set $\mathcal D_E(\R^n)$ is dense in $H^1(\R^n)$ where
    \begin{equation} \label{VanishingTestFunctionDef}
      \mathcal D_E(\R^n) := \{u \in \mathcal D(\R^n) : u \text{ vanishes  in an open neighbourhood of } E \}; and
    \end{equation}
  \item \label{lem:zerocapClaim2}
    the set
    \begin{equation*}
      \mathcal D_E(\overline{\Omega}) := \{ u \rst{\Omega} : u \in   \mathcal D_E(\R^n) \}
    \end{equation*}
    is dense in $H^1(\Omega)$.
  \end{enumerate}
  \end{lemma}
  \begin{proof}
    Claim \eqref{lem:zerocapClaim1}:
  Let $u\in H^1(\R^n)$ and $\varepsilon>0$. Then by
  \cite[Theorem 1.4.2.1]{PG:EPNSD} there is $v \in \mathcal D(\R^n)$
  such that $\norm{u-v}_{H^1(\R^n)}<\varepsilon/2$.
  
  By the vanishing capacity assumption \eqref{CapacityDef}, there is a
  sequence $\{\varphi_j \}_{j=1,2,\ldots} \subset C^\infty(\R^n)$ such
  that $\varphi_j\rst{N_j}=1$ for some neighbourhoods $N_j$ of $E$,
  and also
  \begin{equation} \label{ZeroCapLemmaEq1}
    \lim_{j\to \infty} {\int_{\R^n}{\left ( \abs{\varphi_j}^2+\abs{\nabla \varphi_j}^2 \right )\dif V}} = 0. 
  \end{equation}  
  Defining $v_j(\br) := v(\br) (1-\varphi_j(\br))$ we see that
  each of these functions
  satisfies $v_j \in \mathcal D_E(\R^n)$.  It remains to prove that
  $\norm{v_j - v}_{H^1(\R^n)}<\varepsilon/2$ for all $j$ large enough,
  since then
  \begin{equation*}
    \norm{v_j - u}_{H^1(\R^n)}\leq \norm{v_j - v}_{H^1(\R^n)} + \norm{u-v}_{H^1(\R^n)}
    <\varepsilon.
  \end{equation*}
   By possibly replacing $\{\varphi_j \}_{j=1,2,\ldots}$ by its subsequence, we
  may assume that $\varphi_j\to 0$ pointwise almost everywhere; see
  \cite[Theorem 3.12]{WR:RCA}.  Because $\abs{v_j(\br)} \leq
  \abs{v(\br)}$ for all $\br \in \R^n$ and $j = 1,2,\ldots$, we have
  $v_j\to v$ in $L^2(\R^n)$ by the Lebesgue dominated convergence
  theorem. For the gradients, we note that $\nabla(v_j -v)= -
  \varphi_j \nabla v - v\nabla\varphi_j$.  Thus $\abs{\nabla(v_j -
    v)}\to 0$ in $L^{2}(\R^n)$, since both $\varphi_j$ and
  $\abs{\nabla\varphi_j}$ tend to zero in $L^2(\R^n)$ by
  \eqref{ZeroCapLemmaEq1}.

Claim \eqref{lem:zerocapClaim2}: Let $u \in H^1(\Omega)$ and take
$\varepsilon>0$.  Since $\Omega$ has a Lipschitz boundary, there is an
extension operator $T \in \BLO(H^{1}(\Omega); H^1(\R^n))$ such that
$\left ( T u \right )\rst{\Omega}=u$; see \cite[Theorem
  1.4.3.1]{PG:EPNSD}.  By claim \eqref{lem:zerocapClaim1}, there is a
function $v \in \mathcal D_E(\R^n)$ such that
  \begin{equation*}
   \norm{u-v\rst{\Omega}}_{H^1(\Omega)} \leq \norm{T u-v}_{H^1(\R^n)}  <\varepsilon
  \end{equation*}
  which completes the proof.
\end{proof}

Let us review the Sobolev spaces and the boundary trace mappings on
$\Omega$ and $\partial \Omega$ when the standing assumptions
\eqref{StandingAss1} -- \eqref{StandingAss4} above hold.  The boundary
Sobolev spaces $H^{s}(\partial \Omega)$ and $H^{s}(\Gamma_j)$ for $s
\in [-1,1]$ are defined as in \cite[Definitions 1.2.1.1 and
1.3.3.2]{PG:EPNSD}.  The zero extension Sobolev spaces on $\Gamma_j$ 
are defined by
\begin{equation*}
 \tilde H^{s}(\Gamma_j) := \{ u \in H^{s}(\Gamma_j) : \tilde u \in
 H^{s}(\partial \Omega) \}
\end{equation*}
for $s \in (0,1]$ where
\begin{equation} \label{ZeroExtendedPartialTrace}
  \tilde u(\br) := 
  \begin{cases} 
    u(\br) & \text{ if } \br \in \Gamma_j \\
    0 & \text{ if } \br \in \partial \Omega \setminus \Gamma_j.
  \end{cases}
\end{equation}
We use the Hilbert space norms $\norm{u}_{\tilde H^{s}(\Gamma_j)} :=
\norm{\tilde u}_{ H^{s}(\partial \Omega)}$. The space $\tilde
H^{s}(\Gamma_j)$ is closed in this norm since restriction to
$\Gamma_j$ from $\partial \Omega$ is a bounded operator from
$H^{s}(\partial \Omega)$ to $H^{s}(\Gamma_j)$ for $0\leq s\leq
1$. This boundedness follows trivially by restriction using the
Gagliardo seminorm, see \cite[Eq. (1,3,3,3) on  p.~20]{PG:EPNSD}.
 Then $H^{s}(\partial \Omega)\subset L^2(\partial \Omega)$ and $\tilde
 H^{s}(\Gamma_j) \subset H^{s}(\Gamma_j) \subset L^2(\Gamma_j)$ with
 bounded inclusions.

The \emph{Dirichlet trace operator} $\gamma$ is first defined for
functions $f \in \mathcal D(\closure{\Omega})$ simply by restriction
$\gamma f := f\rst{\partial \Omega}$.  This operator has a unique
extension to a bounded operator $\gamma \in
\BLO(H^{1}(\Omega);H^{1/2}(\partial \Omega))$; see \cite[Theorem
  1.5.1.3]{PG:EPNSD} and Lemma \ref{lem:zerocap}. All this holds for
any Lipschitz domain $\Omega$.

We define the \emph{Neumann trace operator} separately on each surface
$\Gamma_j$ using the vector fields $\bnu_j$.  Such an operator
$\gamma_j \frac{\partial}{\partial \bnu_j}$ is first defined on
$\mathcal D(\closure{\Omega})$ (with values in $L^2(\partial \Omega))$
by setting $\left ( \gamma_j \frac{\partial}{\partial \bnu_j} f \right
)(\br) := \bnu_j(\br) \cdot \nabla f(\br)$ for all $\br \in \Gamma_j$;
here $\gamma_j f := f\rst{ \Gamma_j}$ and
$\frac{\partial}{\partial \bnu_j} := \bnu_j \cdot \nabla$.  It is easy
to see that $\frac{\partial f}{\partial \bnu_j} \in H^1(\Omega)$ and
hence $\gamma_j \frac{\partial}{\partial \bnu_j}$ has an extension to
an operator in $\BLO(H^2(\Omega);H^{1/2}(\Gamma_j))$ by \cite[Theorem
  1.5.1.3]{PG:EPNSD}.  We then define the full Neumann trace operator
$\gamma \frac{\partial }{\partial \bnu}$ on $\cup_j \Gamma_j$ by
  \begin{equation*}
    \gamma \frac{\partial f}{\partial \bnu}(\br) :=  \gamma_j \frac{\partial f}{\partial \bnu_j}(\br)
    \quad \text{ for all } f \in H^2(\Omega) \quad \text{ and (almost) all } \quad  \br \in \Gamma_j.
  \end{equation*}
  Note that the function $\gamma \frac{\partial f}{\partial \bnu}$ is
  not defined at all on the exceptional set of capacity zero
  \begin{equation} \label{ExeptionalSet}
    E := \cup_{j\not= k}(\overline{\Gamma}_j\cap\overline{\Gamma}_k)
  \end{equation}
  of the non-smooth part of $\partial \Omega$. That $C(E) = 0$ follows
  from the standing assumption \eqref{StandingAss3}  by
  \cite[Theorem 3, p. 154]{E-G:MTFPF}.
  
 We need to extend each $\gamma_j \frac{\partial}{\partial \bnu_j}$ to
 the Hilbert space
 \begin{equation*}
   E(\Delta;L^2(\Omega)) := \{ f \in H^1(\Omega): \Delta f \in L^2(\Omega) \}
 \end{equation*}
 that is equipped with the norm defined by $ \| f
 \|_{E(\Delta;L^2(\Omega))}^{2} = \| f \|_{H^1(\Omega)}^{2} + \|
 \Delta f \|_{L^{2}(\Omega)}^{2}$.

 We use an appropriate $L^2$ space as the pivot space for Sobolev
 spaces and their duals.

\begin{proposition} \label{NeumannTraceOnMaximalDomainProp}
  Let the domain $\Omega \subset \R^n$ satisfy the standing
  assumptions \eqref{StandingAss1} -- \eqref{StandingAss4}.
\begin{enumerate}  
\item \label{NeumannTraceOnMaximalDomainPropClaim1}
  Then each Neumann trace operator $\gamma_j \frac{\partial}{\partial
    \bnu_j}$ (originally defined on $\mathcal D(\overline{\Omega})$)
  has a unique extension (also denoted by $\gamma_j
  \frac{\partial}{\partial \bnu_j}$) that is bounded from
  $E(\Delta;L^2(\Omega))$ into the dual space of $\tilde
  H^{1/2}(\Gamma_j)$.
\item \label{NeumannTraceOnMaximalDomainPropClaim2}
  We have 
  \begin{equation*} 
    \int_{\Omega}{\left ( \Delta u \right ) v\dif V}
    +\int_{\Omega}{\nabla u\cdot \nabla v \dif V} = \sum_j{\left <
      \gamma_j \frac{\partial u}{\partial \bnu}, \gamma_j v \right
      >_{[\tilde H^{1/2}(\Gamma_j)]^d, \tilde H^{1/2}(\Gamma_j)}}
  \end{equation*}
for all $u \in E(\Delta;L^2(\Omega))$ and $v \in H^1(\Omega)$ such
that $\gamma_j v \in \tilde H^{1/2}(\Gamma_j)$ for all $j$.
\end{enumerate} 
\end{proposition}
\begin{proof}
  The classical Green's identity for $u \in \mathcal D(\overline{\Omega})$ and $v
  \in \mathcal D_E(\overline{\Omega})$ is
  \begin{equation} \label{ClassicalGreenId} 
    \int_{\Omega}{\left ( \Delta u \right ) v\dif V} 
    +\int_{\Omega}{\nabla u\cdot \nabla v \dif V}
    = \sum_{j}{\int_{\Gamma_j} { \gamma_j \frac{\partial u}{\partial  \bnu_j} \gamma_j v \dif A }}, 
  \end{equation}
  where $E$ is the exceptional set in \eqref{ExeptionalSet}.  Indeed,
  since $v$ vanishes near the interfaces
  $\overline{\Gamma}_j\cap\overline{\Gamma}_k$ for $j \neq k$, we may
  initially apply Green's identity just like \eqref{ClassicalGreenId}
  but over a subdomain of $\Omega$ that has been obtained from
  $\Omega$ by rounding slightly at all $\partial \Gamma_j$'s but
  preserving essentially all of $\partial \Omega$.  Then we get
  \eqref{ClassicalGreenId} by rewriting the result as integrals over
  the original $\Omega$ and the original boundary pieces $\Gamma_j$,
  noting that on additional points the integrands vanish because $ v
  \in \mathcal D_E(\overline{\Omega})$.

 It follows from \eqref{ClassicalGreenId} that we have for $u \in
 \mathcal D(\overline{\Omega})$ and $v \in \mathcal
 D_E(\overline{\Omega})$ the estimate
 \begin{equation} \label{GreenEstimate}
   \abs{\sum_{j}{ \left <\gamma_j \frac{\partial u}{\partial \bnu_j},
       \gamma_j v \right >_{L^2(\Gamma_j)}}} \leq
   \norm{u}_{E(\Delta;L^2(\Omega))} \cdot 4 \norm{v}_{H^1(\Omega)}.
 \end{equation}
 Because $\mathcal D_E(\overline{\Omega})$ is dense in $H^1(\Omega)$
 by Lemma \ref{lem:zerocap} and $\gamma \in
 \BLO(H^1(\Omega);H^{1/2}(\partial \Omega))$ by the trace theorem
 \cite[Theorem 1.5.1.3]{PG:EPNSD}, we conclude that
 \eqref{GreenEstimate} holds for all $u \in \mathcal
 D(\overline{\Omega})$ and $v \in H^1(\Omega)$.

 Fix now $j$ and $g \in \tilde H^{1/2}(\Gamma_j)$, and define $\tilde
 g \in H^{1/2}(\partial \Omega)$ by \eqref{ZeroExtendedPartialTrace}.
 Because the Dirichlet trace $\gamma: H^1(\Omega) \to H^{1/2}(\partial
 \Omega)$ is bounded and surjective, it has a continuous right inverse
 $P \in \BLO(H^{1/2}(\partial \Omega); H^1(\Omega))$, see
 \cite[Theorem 1.5.1.3]{PG:EPNSD}.  Thus there exists $v \in
 H^1(\Omega)$ such that $\gamma_j v = \tilde g \rst{\Gamma_j} = g$ and
 $\gamma_k v = 0$ for $k \neq j$; we may choose $v=P\tilde g$.  From
 this, we have the estimate $4 \norm{v}_{H^1(\Omega)} \leq K
 \norm{\tilde g}_{H^{1/2}(\partial \Omega)} = K \norm{ g}_{\tilde
   H^{1/2}(\Gamma_j)}$.

  It follows from all this and \eqref{GreenEstimate} that we have
  \begin{equation} \label{ContLinFunct}
    \abs{\Phi_g(u) } \leq K \norm{u}_{E(\Delta;L^2(\Omega))} \cdot
    \norm{g}_{\tilde H^{1/2}(\Gamma_j)}
  \end{equation}
  for all $g \in \tilde H^{1/2}(\Gamma_j)$ where $\Phi_g(u) := \left
  <\gamma \frac{\partial u}{\partial \bnu}, \tilde g \right
  >_{L^2(\partial \Omega)} = \left <\gamma_j \frac{\partial u}{\partial
    \bnu_j}, g \right >_{L^2(\Gamma_j)}$ for $u \in \mathcal
  D(\overline{\Omega})$.  Since $\mathcal D(\overline{\Omega})$ is dense
  in $E(\Delta;L^2(\Omega))$ by \cite[Lemma 1.5.3.9]{PG:EPNSD}, we may
  extend $\Phi_g$, $g \in \tilde H^{1/2}(\Gamma_j)$, by continuity to a
  continuous linear functional on $E(\Delta;L^2(\Omega))$ satisfying
  estimate \eqref{ContLinFunct}, too.
  
  For each fixed $u \in E(\Delta;L^2(\Omega))$, the mapping $g \mapsto
  \Phi_g(u)$ is a continuous linear functional on $\tilde
  H^{1/2}(\Gamma_j)$ by \eqref{ContLinFunct}.  Hence, there is a
  representing vector -- denoted by $\gamma_j \frac{\partial u}{\partial
    \bnu_j}$ -- in the dual space $[\tilde H^{1/2}(\Gamma_j)]^d$ such
  that $\Phi_g(u) = \left <\gamma_j \frac{\partial u}{\partial \bnu_j},
  g \right >_{[\tilde H^{1/2}(\Gamma_j)]^d, \tilde H^{1/2}(\Gamma_j)}$.
  This proves claim \eqref{NeumannTraceOnMaximalDomainPropClaim1}.
  Claim \eqref{NeumannTraceOnMaximalDomainPropClaim2} follows by a
  density argument using claim
  \eqref{NeumannTraceOnMaximalDomainPropClaim1} and
  \eqref{ContLinFunct}.
\end{proof}


\begin{theorem}[Green's identity]
 \label{GreensIdentityThm}
 Let the domain $\Omega \subset \R^n$ satisfy the standing assumptions
 \eqref{StandingAss1} -- \eqref{StandingAss4} above.  Assume that
 $u\in H^1(\Omega)$ is such that $\Delta u\in L^2(\Omega)$ and
 satisfies $\frac{\partial u}{\partial\bnu}\in
 L^2(\cup_{j=1}^k\Gamma_j)$ for some $1\leq k\leq n$. Then the Green's
 identity
  \begin{equation}
    \label{GreenIdentity}
    \int_{\Omega} \left ( \Delta u \right ) v\dif V+\int_{\Omega}\nabla u\cdot \nabla
    v\dif V
    =\sum_{j=1}^k\int_{\Gamma_j}\frac{\partial u}{\partial
      \bnu}v\dif A
    +\sum_{j=k+1}^n\left <\gamma_j \frac{\partial u}{\partial \bnu_j},
  \gamma_j v \right >_{[\tilde H^{1/2}(\Gamma_j)]^d, \tilde H^{1/2}(\Gamma_j)}
  \end{equation}
  holds for functions $v\in H^1(\Omega)$ such that $\gamma_j v\in
  \tilde H^{1/2}(\Gamma_j)$ for $k+1\leq j\leq n$.
\end{theorem}
\noindent For $n = 2$, this is a generalisation of \cite[Theorem
1.5.3.11]{PG:EPNSD}.  See also \cite[discussion on p. 62]{PG:EPNSD}
for domains with $C^{1,1}$-boundaries.  The assumption $\frac{\partial
  u}{\partial\bnu}\in L^2(\cup_{j=1}^k\Gamma_j)$ simply means that
$\gamma_j \frac{\partial u}{\partial\bnu_j}\in L^2(\Gamma_j)$ for all
$j=1,2,\ldots,k$ where $\gamma_j \frac{\partial u}{\partial\bnu_j}$ is
understood as an element of $[\tilde H^{1/2}(\Gamma_j)]^d$ which space
includes $L^2(\Gamma_j)$; see Proposition
\ref{NeumannTraceOnMaximalDomainProp}. 
\begin{proof}
  As explained above, we have $\gamma_j v, \gamma_j \frac{\partial
    u}{\partial\bnu_j}\in L^2(\Gamma_j)$ for all $j=1,\ldots,k$.
Then
\eqref{GreenIdentity} follows from claim
\eqref{NeumannTraceOnMaximalDomainPropClaim2} of Proposition
\ref{NeumannTraceOnMaximalDomainProp} under the additional assumption
that $\gamma_j v \in \tilde H^{1/2}(\Gamma_j)$ for all $j$. The
functions in $\mathcal D_E(\overline{\Omega})$ clearly satisfy this
additional assumption, and they are dense in $H^1(\Omega)$.  This
proves the claim.
\end{proof}
An alternative to the above piecewise construction is to start with
the global Neumann trace $\gamma \frac{\partial }{\partial \nu} u$
defined for $u \in E(\Delta;L^2(\Omega))$ with values in
$H^{-1/2}(\partial \Omega)$, see, e.g.,
\cite[Theorem~13.6.9]{T-W:OCOS}. The global Neumann trace $\gamma
\frac{\partial }{\partial \nu} u$ can be restricted to the spaces
$\tilde H^{1/2}(\Gamma_j) $, and claim
\eqref{NeumannTraceOnMaximalDomainPropClaim2} of Proposition
\ref{NeumannTraceOnMaximalDomainProp} follows from a global Green's
identity in a general Lipschitz domain. However, one still needs Lemma
\ref{lem:zerocap} to prove Theorem \ref{GreensIdentityThm}.

It remains to prove the Poincar\'e inequality that is used to show
that the expression \eqref{EnergyNormEq} is a valid Hilbert space norm
for the state space. Let $\Gamma_j$ be one of the boundary components
of $\partial \Omega$ as described above. By the standing assumptions
\eqref{StandingAss1} and \eqref{StandingAss2} given in the beginning
of this appendix, the set $\Gamma_j$ has a finite, positive area $A_j
= \int_{\Gamma_j}{ \dif A}$.  Thus, we can
define the mean value operator $M_j:H^1(\Omega)\to \C$ on $\Gamma_j$ by
\begin{equation*}
  M_j u=\frac{1}{A_j}\int_{\Gamma_j} \gamma_j u \dif A,
\end{equation*}
It is clear that $M_j$ is a bounded linear functional on
$H^1(\Omega)$, and we may regard it as an element of
$\BLO(H^1(\Omega))$ safistying $M_j^2 = M_j$ by considering $M_j u$ as
a constant function on $\Omega$.
\begin{theorem}[Poincar\'e inequality]
\label{PoincareInEqThm}
Let the domain $\Omega \subset \R^n$ satisfy the standing assumptions
\eqref{StandingAss1} -- \eqref{StandingAss4} above, and let $\Gamma_j$
be one of the boundary components of $\partial \Omega$.  There is a
constant $C < \infty$ such that 
\begin{equation} \label{PoincareInEqThmEq1}
  \norm{u-M_j u}_{L^2(\Omega)}\leq C \norm{\nabla u}_{L^2(\Omega)}
\end{equation}
for all $u\in H^1(\Omega)$.  Thus, we have $\norm{u}_{L^2(\Omega)}\leq
C \norm{\nabla u}_{L^2(\Omega)}$ for $u \in H^1(\Omega) \cap
\Null{\gamma_j}$.
\end{theorem}
\begin{proof} The argument is a standard argument by contradiction
  using the Rellich--Kondrachov compactness theorem, see
  e.g. \cite[Theorem 1, p. 144]{E-G:MTFPF}).  For a contradiction
  against \eqref{PoincareInEqThmEq1}, assume that there exist
  functions $u_k\in H^1(\Omega)$ such that there is the strict
  inequality 
  \begin{equation*}
    \norm{u_k-M_j u_k}_{L^2(\Omega)}> k\norm{\nabla
    u_k}_{L^2(\Omega)}\quad\text{for}\quad k=1,2,\ldots.
  \end{equation*}
  None of the functions $u_k$ are constant functions since for such
  functions \eqref{PoincareInEqThmEq1} holds for any $C \geq 0$.  So,
  we can define the functions
\begin{equation*}
  v_k :=\frac{u_k-M_j u_k}{\norm{u_k-M_j u_k}_{L^2(\Omega)}}
\end{equation*}
satisfying for all $k$ the normalisation $\norm{v_k}_{L^2(\Omega)}= 1$
and also $M_j v_k = 0$ by using $M_j^2 = M_j$. Since
\begin{equation*}
  \norm{\nabla v_k}^2=\frac{\norm{\nabla u_k}^2_{L^2(\Omega)}}{\norm{u_k-M_ju_k}^2_{L^2(\Omega)}}<\frac{1}{k^2}
\end{equation*}
by the counter assumption, we get
\begin{equation*}
    \norm{v_k}_{H^1(\Omega)}^2 = \norm{v_k}_{L^2(\Omega)}^2 +
    \norm{\nabla v_k}_{L^2(\Omega)}^2  \leq  1 + \frac{1}{k^2} \leq 2.
\end{equation*}
Since the embedding $H^1(\Omega) \subset L^2(\Omega)$ is compact (by
the boundedness of $\Omega$ and the Rellich--Kondrachov compactness
theorem, see e.g. \cite[Theorem 1, p. 144]{E-G:MTFPF}), we have a
function $v$ such that $v_k \to v$ in $L^2(\Omega)$ by possibly
replacing $\{ v_k \}$ by its subsequence. Moreover,
$\norm{v}_{L^2(\Omega)}=1$ since $\norm{v_k}_{L^2(\Omega)}=1$ for all
$k$.

Since $\norm{ \nabla v_k}_{L^2(\Omega)}\leq 1/k$, we see that $v_k\to
v$ in $H^1(\Omega)$ and hence $\nabla v=0$. Thus $v$ is a constant
function. Because $M_j v = \lim_{k \to \infty}{M_j v_k} = 0$, we
conclude that $v = 0$ which contradicts the fact that
$\norm{v}_{L^2(\Omega)}=1$. This proves \eqref{PoincareInEqThmEq1},
and the Poincar\'e equality follows trivially from this.
\end{proof}
\end{document}